\documentclass{amsart}
\usepackage[latin9]{inputenc}
\usepackage{amsthm}
\usepackage{amstext}
\usepackage{amssymb}
\usepackage{esint}
\usepackage[unicode=true,
 bookmarks=true,bookmarksnumbered=false,bookmarksopen=false,
 breaklinks=false,pdfborder={0 0 1},backref=false,colorlinks=false]
 {hyperref}
\hypersetup{pdftitle={Elliptic regularity theory applied to time harmonic anisotropic Maxwell's equations with less than Lipschitz complex coefficients},
 pdfauthor={Giovanni S. Alberti and Yves Capdeboscq}}

\makeatletter
\theoremstyle{plain}
\newtheorem{thm}{\protect\theoremname}
  \theoremstyle{plain}
  \newtheorem{prop}[thm]{\protect\propositionname}
  \theoremstyle{plain}
  \newtheorem{lem}[thm]{\protect\lemmaname}
  \theoremstyle{remark}
  \newtheorem{rem}[thm]{\protect\remarkname}

\global\long\def\ii{\mathbf{i}}
\global\long\def\C{\mathbb{C}}

\global\long\def\R{\mathbb{R}}

\global\long\def\phi{\varphi}
\global\long\def\epsilon{\varepsilon}
\global\long\def\div{{\rm div}}
\global\long\def\curl{{\rm curl}}

\global\long\def\Cl{{C}}

\global\long\def\maxa{\Lambda}
\global\long\def\Hcurl{H(\curl,\Omega)}
\global\long\def\Hdiv{H(\div,\Omega)}
\global\long\def\Domega{\mathcal{D}(\Omega)}
\global\long\def\tepsi{\varepsilon^{T}}
\global\long\def\tmuu{\mu^{T}}

\global\long\def\k{\omega}

\global\long\def\freqset{\in\C\setminus\{0\}}
\global\long\def\ve{\eta}

\global\long\def\ik{\ii\k^{-1}}

\usepackage{enumitem}
\setlist{leftmargin=*}

\makeatother

  \providecommand{\lemmaname}{Lemma}
  \providecommand{\propositionname}{Proposition}
  \providecommand{\remarkname}{Remark}
\providecommand{\theoremname}{Theorem}

\begin{document}
\title[Elliptic Regularity for Maxwell's Equations]{Elliptic regularity theory applied to time harmonic anisotropic Maxwell's equations with less than Lipschitz complex coefficients}

\author{Giovanni S. Alberti}
\address{Mathematical Institute, University of Oxford, Andrew Wiles Building, Radcliffe Observatory Quarter, Woodstock Road, Oxford OX2 6GG, UK} \email{giovanni.alberti@maths.ox.ac.uk}

\author{Yves Capdeboscq}
\address{Mathematical Institute, University of Oxford, Andrew Wiles Building, Radcliffe Observatory Quarter, Woodstock Road, Oxford OX2 6GG, UK} \email{yves.capdeboscq@maths.ox.ac.uk}

\thanks{The authors are supported by the EPSRC Science \& Innovation Award to the Oxford Centre for Nonlinear PDE (EP/EO35027/1).}

\begin{abstract}
The focus of this paper is the study of the regularity properties
of the time harmonic Maxwell's equations with anisotropic complex
coefficients, in a bounded domain with $\Cl^{1,1}$ boundary. We assume
that at least one of the material parameters is $W^{1,p}$ for some
$p>3$. Using regularity theory for second order elliptic partial
differential equations, we derive $W^{1,p}$ estimates and Hölder
estimates for electric and magnetic fields up to the boundary, together
with their higher regularity counterparts. We also derive interior
estimates in bi-anisotropic media. 
\end{abstract}


\date{December 28, 2013}

\maketitle

\section{\label{sec:Introduction}Introduction}

Let $\Omega\subseteq\mathbb{R}^{3}$ be a bounded and connected open
set in $\mathbb{R}^{3}$, with $\Cl^{1,1}$ boundary. Let $\varepsilon,{\mu}\in L^{\infty}\left(\Omega;\mathbb{C}^{3\times3}\right)$
be two bounded complex matrix-valued functions with uniformly positive
definite real parts and symmetric imaginary parts. In other words,
there exists a constant $\Lambda>0$ such that for any $\lambda\in\mathbb{C}^{3}$
there holds 
\begin{equation}
2\Lambda\left|\lambda\right|^{2}\le\overline{\lambda}\cdot\left(\varepsilon+\overline{\varepsilon}^{T}\right)\lambda,\,2\Lambda\left|\lambda\right|^{2}\le\overline{\lambda}\cdot\left({\mu}+\overline{\mu}^{T}\right)\lambda\mbox{ and }\left|{\mu}\right|+\left|\varepsilon\right|\le\maxa^{-1}\mbox{ a.e. in }\ensuremath{\Omega},\label{eq:Hyp-Ellip}
\end{equation}
 where $a^{T}$ is the transpose of $a$, $\overline{a}=\Re(a)-\ii\Im(a)$,
where $\ii^{2}=-1$, and $|x|=\mbox{Trace}(\overline{x}^{T}x)$ is
the Euclidean norm. The $3\times3$ matrix $\varepsilon$ represents
the complex electric permittivity of the medium $\Omega$: its real
part is the physical electric permittivity, whereas its imaginary
part is proportional to the electric conductivity, by Ohm's Law. The
$3\times3$ matrix $\mu$ stands for the complex magnetic permeability:
the imaginary part may model magnetic dissipation or lag time.

For a given frequency $\k\freqset$ and current sources $J_{e}$ and
$J_{m}$ in $L^{2}\left(\Omega;\mathbb{C}^{3}\right)$ we are interested
in the regularity of the time-harmonic electromagnetic fields $E$
and $H$, that is, the weak solutions $E$ and $H$ in $\Hcurl$ of
the time-harmonic anisotropic Maxwell's equations 
\begin{equation}
\left\{ \begin{array}{l}
\curl H=\ii\k\varepsilon E+J_{e}\qquad\mbox{ in }\Omega,\\
\curl E=-\ii\k{\mu}H+J_{m}\qquad\mbox{ in }\Omega,\\
E\times\nu=G\times\nu\mbox{ on }\partial\Omega.
\end{array}\right.\label{eq:maxwell}
\end{equation}
 The boundary constraint is meant in the sense of traces, with $G\in H\left(\curl,\Omega\right)$.
Our focus is the dependence of the regularity of $E$ and $H$ on
the coefficients $\varepsilon$ and $\mu$, the current sources $J_{e}$
and $J_{m}$, and the boundary condition $G$. The precise dependence
on the regularity of the boundary of $\Omega$ is beyond the scope
of this work. We refer the reader to \cite{AMROUCHE-BERNARDI-DAUGE-GIRAULT-1998,COSTABEL-DAUGE-2000,BUFFA-COSTABEL-SHEEN-2002}
where domains with rougher boundaries are considered. For $N\in\mathbb{N}^{*}$
and $p>1$ we denote by $W^{N,p}\left(\curl,\Omega\right)$ and $W^{N,p}\left(\div,\Omega\right)$
the Banach spaces 
\begin{eqnarray*}
 & W^{N,p}\left(\curl,\Omega\right)=\bigl\{ v\in W^{N-1,p}\left(\Omega;\mathbb{C}^{3}\right):\curl v\in W^{N-1,p}(\Omega;\mathbb{C}^{3})\bigr\},\\
 & W^{N,p}\left(\div,\Omega\right)=\bigl\{ v\in W^{N-1,p}\left(\Omega;\mathbb{C}^{3}\right):\div v\in W^{N-1,p}(\Omega;\mathbb{C})\bigr\},
\end{eqnarray*}
 equipped with canonical norms. The space $W^{1,2}\left(\curl,\Omega\right)$
is the space $H\left(\curl,\Omega\right)$ mentioned above, and $W^{1,2}\!\left(\div,\Omega\right)$
is commonly denoted by $H\left(\div,\Omega\right)$. Throughout this
paper, $H^{1}(\Omega)=W^{1,2}\left(\Omega;\mathbb{C}^{3}\right)$
and $L^{2}(\Omega)=L^{2}\left(\Omega;\mathbb{C}^{3}\right)$.

It is very well known that when the domain is a cylinder $\Omega^{\prime}\times(0,L)$,
the electric field $E$ has only one component, $E=(0,0,u)^{T}$,
the physical parameters are real, scalar and do not depend on the
third variable, then $u$ satisfies a second order elliptic equation
in the first two variables 
\[
\div\left({\mu}^{-1}\nabla u\right)+\k^{2}\varepsilon u=0\mbox{ in }\Omega^{\prime}.
\]
 In such a case, the regularity of $u$ follows from classical elliptic
regularity theory. In particular, $u$ is Hölder continuous due to
the De Giorgi--Nash Theorem (at least in the interior). The regularity
of $E$ and $H$ is less clear when the material parameters are anisotropic
and/or complex valued. For general non diagonal elliptic systems with
non regular coefficients, Müller and {Š}ver{á}k \cite{MULLER-SVERAK-2003}
have shown that the solutions may not be in $W^{1,2+\delta}$ for
any $\delta>0$ . Assuming that the coefficients are real, anisotropic,
suitably smooth matrices, Leis \cite{LEIS-1968} established well-posedness
in $H^{1}(\Omega)$. The regularity of the coefficients was reduced
to globally Lipschitz in Weber \cite{WEBER-1981}, for a $\Cl^{2}$
smooth boundary, and $\Cl^{1}$ for a $\Cl^{1,1}$ domain in Costabel
\cite{COSTABEL-1991}.

As far as the authors are aware, neither the $H^{1}$ nor the Hölder
regularity of the electric and magnetic fields for complex anisotropic
less than Lipschitz media have been addressed so far. Anisotropic
dielectric parameters have received a renewed attention in the last
decades. They appear for example in the mathematical theory of liquid
crystals, in optically chiral media, and in meta-materials. In this
work we show that the theory of elliptic boundary value problems can
be used to study the general case of complex anisotropic coefficients.

Our first result addresses the $H^{1}(\Omega)$ regularity of $E$. 
\begin{thm}
\label{thm:H1 for E}Assume that \eqref{eq:Hyp-Ellip} holds, and
that $\varepsilon$ also satisfies 
\begin{equation}
\varepsilon\in W^{1,3+\delta}\left(\Omega;\mathbb{C}^{3\times3}\right)\mbox{ for some }\delta>0.\label{eq:Hyp2-epsilon}
\end{equation}
 Suppose that the source terms $J_{m},\, J_{e}$ and $G$ satisfy
\begin{equation}
J_{m}\in L^{p}\left(\Omega;\C^{3}\right),\, J_{e}\in W^{1,p}\left(\div,\Omega\right)\mbox{ and }G\in W^{1,p}\left(\Omega;\mathbb{C}^{3}\right),\label{eq:reg-E}
\end{equation}
 for some $p\geq2.$ If $E,H\in H\left(\curl,\Omega\right)$ are weak
solutions of \eqref{eq:maxwell}, then $E\in H^{1}\left(\Omega\right)$
and 
\begin{equation}
\left\Vert E\right\Vert _{H^{1}\left(\Omega\right)}\le C\bigl(\left\Vert E\right\Vert _{H\left(\curl,\Omega\right)}+\left\Vert G\right\Vert _{H^{1}\left(\Omega\right)}+\left\Vert J_{m}\right\Vert _{L^{2}\left(\Omega\right)}+\left\Vert J_{e}\right\Vert _{\Hdiv}\bigr),\label{eq:bound norm E in H^1}
\end{equation}
 for some constant $C$ depending on $\Omega$, $\Lambda$ given in
\eqref{eq:Hyp-Ellip}, $\k$ and $\left\Vert \varepsilon\right\Vert _{W^{1,3+\delta}\left(\Omega;\mathbb{C}^{3\times3}\right)}$
only. 
\end{thm}
Note that no regularity assumption is made on $\mu$, apart from \eqref{eq:Hyp-Ellip}.
Our second result is devoted to the $H^{1}(\Omega)$ regularity of
$H$. 
\begin{thm}
\label{thm:H1 for H} Assume that \eqref{eq:Hyp-Ellip} holds, and
that $\mu$ also satisfies 
\begin{equation}
\mu\in W^{1,3+\delta}\left(\Omega;\mathbb{C}^{3\times3}\right)\mbox{ for some }\delta>0.\label{eq:Hyp2-mu}
\end{equation}
 Suppose that the source terms $J_{e},\, J_{m}$ and $G$ satisfy
\begin{eqnarray}
 & J_{e}\in L^{p}\left(\Omega;\C^{3}\right),\, J_{m}\in W^{1,p}\left(\div,\Omega\right),\, J_{m}\cdot\nu\in W^{1-\frac{1}{p},p}\left(\partial\Omega;\C\right)\nonumber \\[-1.5ex]
\label{eq:reg-H}\\[-1.5ex]
 & \mbox{ and }G\in W^{1,p}\left(\Omega;\mathbb{C}^{3}\right),\nonumber 
\end{eqnarray}
 for some $p\geq2.$ If $E,H\in H\left(\curl,\Omega\right)$ are weak
solutions of \eqref{eq:maxwell}, then $H\in H^{1}\left(\Omega\right)$
and
\begin{multline*}
\left\Vert H\right\Vert _{H^{1}\left(\Omega\right)}\le C\left(\left\Vert H\right\Vert _{H\left(\curl,\Omega\right)}+\left\Vert G\right\Vert _{H^{1}\left(\Omega\right)}\right.\\
+\left.\left\Vert J_{e}\right\Vert _{L^{2}\left(\Omega\right)}+\left\Vert J_{m}\right\Vert _{\Hdiv}+\left\Vert J_{m}\cdot\nu\right\Vert _{H^{1/2}(\partial\Omega;\C)}\right),
\end{multline*}
for some constant $C$ depending on $\Omega$, $\Lambda$ given in
\eqref{eq:Hyp-Ellip}, $\k$ and $\left\Vert \mu\right\Vert _{W^{1,3+\delta}\left(\Omega;\mathbb{C}^{3\times3}\right)}$
only. 
\end{thm}
Naturally, interior regularity for $H$ follows from the interior
regularity of $E$, due to the (almost) symmetrical role of the pairs
$(E,\varepsilon)$ and $(H,\mu)$ in Maxwell's equations. The difference
between Theorem~\ref{thm:H1 for E} and Theorem~\ref{thm:H1 for H}
comes from the fact that \eqref{eq:maxwell} involves a boundary condition
on $E$, not on $H$. Combining both results, we then show that when
$\varepsilon$ and $\mu$ are both $W^{1,3+\delta}$ with $\delta>0$,
then $E$ and $H$ enjoy the regularity inherited from the source
terms, up to $W^{1,3+\delta}$. 
\begin{thm}
\label{thm:W^1,p for E and H} Suppose that the hypotheses of Theorems~\ref{thm:H1 for E}
and \ref{thm:H1 for H} hold.

If $E$ and $H$ in $H\left(\curl,\Omega\right)$ are weak solutions
of \eqref{eq:maxwell}, then $E,H\in W^{1,q}\left(\Omega;\mathbb{C}^{3}\right)$
with $q=\min\left(p,3+\delta\right)$ and
\begin{multline*}
\left\Vert E\right\Vert _{W^{1,q}\left(\Omega;\mathbb{C}^{3}\right)}+\left\Vert H\right\Vert _{W^{1,q}\left(\Omega;\mathbb{C}^{3}\right)}\le C\Bigl(\left\Vert E\right\Vert _{L^{2}\left(\Omega\right)}+\left\Vert H\right\Vert _{L^{2}\left(\Omega\right)}+\left\Vert G\right\Vert _{W^{1,p}\left(\Omega;\mathbb{C}^{3}\right)}\Bigr.\\
+\Bigl.\left\Vert J_{e}\right\Vert _{W^{1,p}\left(\div,\Omega\right)}+\left\Vert J_{m}\right\Vert _{W^{1,p}\left(\div,\Omega\right)}+\left\Vert J_{m}\cdot\nu\right\Vert _{W^{1-\frac{1}{p},p}(\partial\Omega;\C)}\Bigr),
\end{multline*}
for some $C>0$ depending on $\Omega$, $\Lambda$, $\k$, $q$, $\left\Vert \varepsilon\right\Vert _{W^{1,3+\delta}\left(\Omega;\mathbb{C}^{3\times3}\right)}$
and $\left\Vert \mu\right\Vert _{W^{1,3+\delta}\left(\Omega;\mathbb{C}^{3\times3}\right)}$
only. In particular, if $p>3,$ then $E,H\in\Cl^{0,\alpha}\left(\overline{\Omega};\mathbb{C}^{3}\right)$
with $\alpha=\min\bigl(1-\frac{3}{p},\frac{\delta}{3+\delta}\bigr)$. 
\end{thm}
As an extension of this work, we show in \S~\ref{sec:Bi-anisotropic}
that, as far as interior regularity is concerned, the analogue of
Theorem \ref{thm:W^1,p for E and H} holds for more general constitutive
relations, for which Maxwell's equations read 
\begin{equation}
\left\{ \begin{array}{l}
\curl H=\ii\k\left(\varepsilon E+\xi H\right)+J_{e}\qquad\mbox{ in }\Omega,\\
\curl E=-\ii\k\left(\zeta E+\mu H\right)+J_{m}\qquad\mbox{ in }\Omega,
\end{array}\right.\label{eq:maxwell for bi-anisotropic}
\end{equation}
 provided that $\zeta,\xi\in L^{\infty}\left(\Omega;\mathbb{C}^{3\times3}\right)$
are small enough to preserve the underlying elliptic structure of
the system -- see condition \eqref{eq:strong legendre condition}.
These constitutive relations are commonly used to model the so called
bi-anisotropic materials.

Our approach is classical and fundamentally scalar. It is oblivious
of the fact that Maxwell's equations is posed on vectors, as we consider
the problem component per component, just like it is done in Leis
\cite{LEIS-1986}. A general $L^{p}$ theory for vector potentials
has been developed very recently by Amrouche \& Seloula \cite{AMROUCHE-SELOULA-2011,AMROUCHE-SELOULA-2013}.
Applying their results would lead to similar regularity results for
scalar coefficients. It seems our approaches are completely independent,
even though both are based on the $L^{p}$ theory for elliptic equations.

Finally, \S~\ref{sec:Campanato} is devoted to the case when only
one of the two coefficients is complex-valued. We consider the case
when $\varepsilon\in W^{1,3+\delta}\left(\Omega;\mathbb{C}^{3\times3}\right)$,
with $\delta>0$, and $\mu\in L^{\infty}(\Omega,\mathbb{R}^{3\times3})$.
In that situation, a Helmholtz decomposition of the magnetic field
into $H=T+\nabla h$, where $T\in H^{1}\left(\Omega\right)$ is divergence
free, provides additional insight on the regularity of $H$. Indeed,
the potential $h$ then satisfies a real scalar second order elliptic
equation, and therefore enjoys additional regularity properties. 
\begin{thm}
\label{thm:global campanato for E}Suppose that the hypotheses of
Theorem~\ref{thm:H1 for E} hold for some $p>3$. Assume additionally
that $\Omega$ is simply connected and that $\Im{\mu}=0$.

If $E$ and $H$ in $H\left(\curl,\Omega\right)$ are weak solutions
of \eqref{eq:maxwell}, then there exists $0<\alpha\leq\min(1-\frac{3}{p},\frac{\delta}{3+\delta})$
depending only on $\Omega$ and $\Lambda$ given in \eqref{eq:Hyp-Ellip}
such that $E\in\Cl^{0,\alpha}(\overline{\Omega};\C^{3})$ with 
\[
\left\Vert E\right\Vert _{\Cl^{0,\alpha}(\overline{\Omega};\C^{3})}\le C\bigl(\left\Vert E\right\Vert _{L^{2}\left(\Omega\right)}+\left\Vert G\right\Vert _{W^{1,p}\left(\Omega;\mathbb{C}^{3}\right)}+\left\Vert J_{e}\right\Vert _{W^{1,p}\left(\div,\Omega\right)}+\left\Vert J_{m}\right\Vert _{L^{p}(\Omega;\C^{3})}\bigr),
\]
 for some constant $C$ depending on $\Omega$, $\Lambda$, $\k$
and $\left\Vert \varepsilon\right\Vert _{W^{1,3+\delta}\left(\Omega;\mathbb{C}^{3\times3}\right)}$
only. 
\end{thm}
This is a generalization of the result proved by Yin~\cite{YIN-2004}
who assumed instead $\varepsilon\in W^{1,\infty}\left(\Omega;\mathbb{C}\right)$
and $\mu\in L^{\infty}\left(\Omega;\mathbb{R}\right)$: this is not
the minimal regularity requirement to prove Hölder continuity of the
electric field.

We do not claim that requiring that (one of) the parameters is in
$W^{1,3+\delta}$ for some $\delta>0$ is optimal. We are confident
that it is sufficient to assume that the derivatives are in the Campanato
space $L^{3,\lambda}$ with $\lambda>0$, for example. However, as
we do not know that these are necessary conditions, it seemed that
such a level of sophistication was unjustified in this work. Assuming
simply $W^{1,3}$ regularity (i.e., $\delta=\lambda=0$) does not
seem to work with our proof: the bootstrap argument we use stalls
in this case. A completely different approach would be required to
handle the case of coefficients with less than VMO regularity.

Our paper is structured as follows. Section~\ref{sec:Main-H1-Ca}
is devoted to the proof of Theorems~\ref{thm:H1 for E}, \ref{thm:H1 for H}
and \ref{thm:W^1,p for E and H}. At the end of \S~\ref{sec:Main-H1-Ca}
we prove Theorem~\ref{thm:higher regularity-new}, the $W^{N,p}$
counterpart of Theorem~\ref{thm:W^1,p for E and H}, with appropriately
smooth coefficients in a domain with $\Cl^{N,1}$ boundary. Section~\ref{sec:Bi-anisotropic}
is devoted to the statement of our result for the generalized bi-anisotropic
Maxwell's equations; the proof of this result is given in the appendix.
Section~\ref{sec:Campanato} focuses on the particular case when
$\mu$ is real-valued and is devoted to the proof of Theorem~\ref{thm:global campanato for E}.

\section{\label{sec:Main-H1-Ca} $W^{1,p}$ regularity for $E$ and $H$}

Our strategy is to consider a coupled elliptic system satisfied by
each component of the electric and magnetic field, where in each equation,
only one component appears in the leading order term. In a first step,
we show that the electric and magnetic fields are very weak solutions
of such a system. This system was already introduced, in its strong
form, in Leis \cite{LEIS-1986}, and was used recently in Nguyen \&
Wang \cite{NGUYEN-WANG-2012}. 
\begin{prop}
\label{prop:coupled_system} Assume that \eqref{eq:Hyp-Ellip} holds
true. Let $E=\left(E_{1},E_{2},E_{3}\right)^{T}$ and $H=\left(H_{1},H_{2},H_{3}\right)^{T}$
in $H\left(\curl,\Omega\right)$ be weak solutions of \eqref{eq:maxwell}. 
\begin{itemize}
\item If \eqref{eq:Hyp2-epsilon} and \eqref{eq:reg-E} hold, for each $k=1,2,3$,
$E_{k}$ is a very weak solution of 
\begin{equation}
-\div\left(\varepsilon\nabla E_{k}\right)=\div\left(\left(\partial_{k}\varepsilon\right)E-\varepsilon\left(\mathbf{e}_{k}\times\left(J_{m}-\ii\k\mu H\right)\right)-\ik\mathbf{e}_{k}\div J_{e}\right)\textrm{ in }\Omega,\label{eq:a}
\end{equation}
 where $\mathbf{e}_{k}$ is the unit vector in the $k$-th direction.
More precisely, $E_{k}$ satisfies for any $\varphi\in W^{2,2}(\Omega;\C)$
\begin{eqnarray}
 & \int_{\Omega}E_{k}\div\left(\tepsi\nabla\overline{\varphi}\right)\, dx=\int_{\partial\Omega}(\partial_{k}\overline{\varphi})\varepsilon E\cdot\nu\, d\sigma-\int_{\partial\Omega}\left(\mathbf{e}_{k}\times\left(E\times\nu\right)\right)\cdot(\tepsi\nabla\overline{\varphi})\, d\sigma\nonumber \\[-0.5ex]
\label{eq:very weak E}\\[-1.5ex]
 & +\int_{\Omega}\left((\partial_{k}\varepsilon)E-\varepsilon\left(\mathbf{e}_{k}\times\left(J_{m}-\ii\k\mu H\right)\right)-\ik\mathbf{e}_{k}\div J_{e}\right)\cdot\nabla\overline{\varphi}\, dx.\nonumber 
\end{eqnarray}

\item If \eqref{eq:Hyp2-mu} and \eqref{eq:reg-H} hold, for each\textup{
$k=1,2,3$, }$H_{k}$ is a very weak solution of 
\begin{equation}
-\div\left({\mu}\nabla H_{k}\right)=\div\left((\partial_{k}{\mu})H-{\mu}\left(\mathbf{e}_{k}\times\left(J_{e}+\ii\k\varepsilon E\right)\right)+\ik\mathbf{e}_{k}\div J_{m}\right)\textrm{ in }\Omega.\label{eq:b}
\end{equation}
 More precisely, $H_{k}$ satisfies for any $\varphi\in W^{2,2}(\Omega;\C)$
\begin{eqnarray}
 & \int_{\Omega}H_{k}\div\left({\mu}^{T}\nabla\overline{\varphi}\right)\, dx=\int_{\partial\Omega}(\partial_{k}\overline{\varphi}){\mu}H\cdot\nu\, d\sigma-\int_{\partial\Omega}\left(\mathbf{e}_{k}\times\left(H\times\nu\right)\right)\cdot({\mu}^{T}\nabla\overline{\varphi})\, d\sigma\nonumber \\[-0.5ex]
\label{eq:very weak H}\\[-1.5ex]
 & +\int_{\Omega}\left((\partial_{k}{\mu})H-{\mu}\left(\mathbf{e}_{k}\times\left(J_{e}+\ii\k\varepsilon E\right)\right)+\ik\mathbf{e}_{k}\div J_{m}\right)\cdot\nabla\overline{\varphi}\, dx.\nonumber 
\end{eqnarray}
 
\end{itemize}
\end{prop}
\begin{proof}
We detail the derivation of \eqref{eq:very weak E} for the sake of
completeness. The derivation of \eqref{eq:very weak H} is similar,
thanks to the intrinsic symmetry of Maxwell's equations \eqref{eq:maxwell}.

We multiply the identity $\curl E=-\ii\k{\mu}H+J_{m}$ by $\overline{\Phi}=\overline{g}\mathbf{e}_{l}$
for some $g\in W^{1,2}(\Omega;\C)$, integrate by parts and multiply
the result by $\mathbf{e}_{l}$. We obtain 
\[
\mathbf{e}_{l}\int_{\Omega}\overline{g}\left(-\ii\k\mu H+J_{m}\right)\cdot\mathbf{e}_{l}\, dx=\mathbf{e}_{l}\int_{\Omega}E\cdot(\nabla\times\overline{\Phi})\, dx-\mathbf{e}_{l}\int_{\partial\Omega}\left(E\times\nu\right)\cdot\overline{\Phi}\, d\sigma,
\]
 which can be written also as 
\[
\int_{\Omega}\overline{g}\left(-\ii\k\mu H+J_{m}\right)\, dx+\int_{\partial\Omega}\overline{g}\left(E\times\nu\right)\, d\sigma=\int_{\Omega}E\times\nabla\overline{g}\, dx.
\]
 Note that since $E\in\Hcurl$ by assumption, $E\times\nu$ is well
defined in $H^{-\frac{1}{2}}\left(\partial\Omega;\mathbb{C}^{3}\right)$
and this formulation is valid. Next, we cross product this identity
with $\mathbf{e}_{k}$, and take the scalar product with $\mathbf{e}_{i}$.
Using the vector identity $a\times\left(b\times c\right)=\left(a\cdot c\right)b-\left(a\cdot b\right)c$
on the right-hand side, we obtain 
\begin{eqnarray}
 & \mathbf{e}_{i}\cdot\int_{\Omega}\overline{g}\mathbf{e}_{k}\times\left(-\ii\k\mu H+J_{m}\right)\, dx+\mathbf{e}_{i}\cdot\int_{\partial\Omega}\overline{g}\mathbf{e}_{k}\times\left(E\times\nu\right)\, d\sigma\nonumber \\[-1.5ex]
\label{eq:H-1b}\\[-1.5ex]
 & =\int_{\Omega}E_{i}\partial_{k}\overline{g}-E_{k}\partial_{i}\overline{g}\, dx,\nonumber 
\end{eqnarray}
 for any $i$ and $k$ in $\{1,2,3\}$ and $g\in W^{1,2}(\Omega;\C)$.
In view of \eqref{eq:Hyp2-epsilon}, we have that $\overline{\varepsilon}^{T}\nabla\varphi\in H^{1}(\Omega)$
for any $\varphi\in W^{2,2}(\Omega;\C)$. Thus, applying \eqref{eq:H-1b}
with $g=\left(\overline{\varepsilon}^{T}\nabla\varphi\right)_{i}$
for any $i=1,2,3$ and $\varphi\in W^{2,2}(\Omega;\C)$ we find that
\begin{multline*}
\int_{\Omega}E_{i}\,\partial_{k}\left(\tepsi\nabla\overline{\varphi}\right)_{i}\, dx=\int_{\Omega}E_{k}\,\partial_{i}\left(\tepsi\nabla\overline{\varphi}\right)_{i}\, dx+\mathbf{e}_{i}\cdot\int_{\partial\Omega}\left(\tepsi\nabla\overline{\varphi}\right)_{i}\mathbf{e}_{k}\times\left(E\times\nu\right)\, d\sigma\\
+\mathbf{e}_{i}\cdot\int_{\Omega}\left(\tepsi\nabla\overline{\varphi}\right)_{i}\mathbf{e}_{k}\times\left(-\ii\k\mu H+J_{m}\right)\, dx.
\end{multline*}
Summing over $i$, this yields
\begin{multline}
\int_{\Omega}E\cdot\partial_{k}\left(\tepsi\nabla\overline{\varphi}\right)\, dx=\int_{\Omega}E_{k}\,\div\left(\tepsi\nabla\overline{\varphi}\right)\, dx\\
+\int_{\partial\Omega}\left(\mathbf{e}_{k}\times\left(E\times\nu\right)\right)\cdot(\tepsi\nabla\overline{\varphi})\, d\sigma+\int_{\Omega}\varepsilon\left(\mathbf{e}_{k}\times\left(-\ii\k\mu H+J_{m}\right)\right)\cdot\nabla\overline{\varphi}\, dx.\label{eq:H-1}
\end{multline}

We then use the second part of Maxwell's equations. We test $\curl H-J_{e}=i\k\varepsilon E$
against $\mbox{\ensuremath{\nabla}}\left(\partial_{k}\overline{\varphi}\right)\frac{1}{\ii\k}$
for $\varphi\in W^{2,2}(\Omega;\C)$ and obtain
\[
\begin{split}\int_{\Omega}\varepsilon E\cdot\partial_{k} & \left(\nabla\overline{\varphi}\right)\, dx=-\ik\int_{\Omega}\curl(H)\cdot\nabla\left(\partial_{k}\overline{\varphi}\right)\, dx+\ik\int_{\Omega}J_{e}\cdot\nabla\left(\partial_{k}\overline{\varphi}\right)\, dx\\
 & =-\ik\!\left(\int_{\partial\Omega}\!(\partial_{k}\overline{\varphi})\curl H\cdot\nu\, d\sigma-\int_{\partial\Omega}\!\left(\partial_{k}\overline{\varphi}\right)J_{e}\cdot\nu\, d\sigma+\!\int_{\Omega}\div J_{e}\partial_{k}\overline{\varphi}\, dx\right)\\
 & =-\ik\left(\ii\k\int_{\partial\Omega}(\partial_{k}\overline{\varphi})\varepsilon E\cdot\nu\, d\sigma+\int_{\Omega}\div J_{e}\partial_{k}\overline{\varphi}\, dx\right).
\end{split}
\]
Since $J_{e}\in H\left(\mbox{div},\Omega\right)$, the boundary term
is well defined. Writing the left-hand side of the above identity
in the form 
\[
\int_{\Omega}\varepsilon E\cdot\partial_{k}\left(\nabla\overline{\varphi}\right)\, dx=\int_{\Omega}E\cdot\partial_{k}\left(\tepsi\nabla\overline{\varphi}\right)\, dx-\int_{\Omega}\left(\partial_{k}\varepsilon\right)E\cdot\nabla\overline{\varphi}\, dx,
\]
 we obtain 
\[
-\int_{\Omega}(\partial_{k}\varepsilon)E\cdot\nabla\overline{\varphi}\, dx+\int_{\Omega}E\cdot\partial_{k}\left(\tepsi\nabla\overline{\varphi}\right)\, dx=\int_{\partial\Omega}(\partial_{k}\overline{\varphi})\varepsilon E\cdot\nu\, d\sigma-\ik\int_{\Omega}\div J_{e}\partial_{k}\overline{\varphi}\, dx
\]
 Inserting this identity in \eqref{eq:H-1} we obtain \eqref{eq:very weak E}. 
\end{proof}
To transform the very weak identities given by Proposition~\ref{prop:coupled_system}
into regular weak formulations, we shall use the following lemma.
Given $r\in(1,\infty)$, we write $r'$ the solution of $\frac{1}{r}+\frac{1}{r'}=1$. 
\begin{lem}
\label{lem:very weak_boundary}Assume that \eqref{eq:Hyp-Ellip} and
\eqref{eq:Hyp2-epsilon} hold. Given $r\geq6/5$, $u\in L^{2}(\Omega;\C)\cap L^{r}(\Omega;\C)$,
$F\in(W^{1,r'}(\Omega;\C))'$, let $B$ be the trace operator given
either by $B\varphi=\varphi$ on $\partial\Omega$ or by $B\varphi=\tepsi\nabla\overline{\varphi}\cdot\nu$
on $\partial\Omega$ for $\varphi\in W^{2,2}(\Omega;\C)$.

If for all $\varphi\in W^{2,2}(\Omega;\C)$ such that $B\varphi=0$
there holds 
\begin{equation}
\int_{\Omega}u\,\div(\tepsi\nabla\overline{\varphi})\, dx=\langle F,\varphi\rangle,\label{eq:weak-a}
\end{equation}
 then $u\in W^{1,r}(\Omega;\C)$ and 
\begin{equation}
\left\Vert \nabla u\right\Vert _{L^{r}(\Omega;\C^{3})}\le C\left\Vert F\right\Vert _{(W^{1,r'}(\Omega;\C))'},\label{eq:bound norm grad u}
\end{equation}
 for some constant $C$ depending on $\Omega$, $\Lambda$ given in
\eqref{eq:Hyp-Ellip}, $\left\Vert \varepsilon\right\Vert _{W^{1,3}\left(\Omega;\mathbb{C}^{3\times3}\right)}$
and $r$ only. \end{lem}
\begin{proof}
We first observe that, since $r\ge6/5$, then both terms of the identity
\eqref{eq:weak-a} are well defined as $W^{2,2}(\Omega;\C)\subset W^{1,6}(\Omega;\C)$
and $\frac{1}{6}+\frac{1}{6/5}=1$. Let $\psi\in\Domega$ be a test
function and fix $i=1,2$ or $3$. Let $\varphi^{*}\in W^{1.2}(\Omega;\C)$
be the unique solution to the problem 
\[
\left\{ \begin{array}{l}
\div(\tepsi\nabla\overline{\varphi^{*}})=\partial_{i}\psi\qquad\mbox{ in \ensuremath{\Omega},}\\
B\phi^{*}=0\qquad\mbox{ on \ensuremath{\partial\Omega}}.
\end{array}\right.
\]
 In the case of the Neumann boundary condition, we add the normalization
condition $\int_{\Omega}\varphi^{*}\, dx=0$. Since $\varepsilon\in W^{1,3}\left(\Omega,\mathbb{C}^{3\times3}\right)$,
it is known \cite[Theorem 1]{AUSCHER-QAFSAOUI-2002} that for any
$q\in(1,\infty)$ there holds 
\begin{equation}
\left\Vert \varphi^{*}\right\Vert _{W^{1,q}(\Omega;\C)}\le C\left\Vert \psi\right\Vert _{L^{q}(\Omega;\C)}\label{eq:q bound}
\end{equation}
 for some $C=C(q,\Omega,\Lambda,\left\Vert \varepsilon\right\Vert _{W^{1,3}\left(\Omega;\mathbb{C}^{3\times3}\right)})>0$.
In particular, $\varphi^{*}\in W^{1,q}(\Omega;\C)$ for all $q<\infty$.
The usual difference quotient argument (see e.g. \cite{GRISVARD-1985,GIAQUINTA-MARTINAZZI-2005})
shows in turn that $\varphi^{*}\in W^{2,2}\left(\Omega;\C\right)$,
as $\psi$ is regular. Thus, by assumption we have 
\[
\left|\int_{\Omega}u\partial_{i}\psi\, dx\right|=\left|\int_{\Omega}u\div(\tepsi\nabla\overline{\varphi^{*}})\, dx\right|=\left|\langle F,\varphi^{*}\rangle\right|\le\left\Vert F\right\Vert _{(W^{1,r'}(\Omega;\C))'}\left\Vert \varphi^{*}\right\Vert _{W^{1,r'}(\Omega;\C)},
\]
 which in view of \eqref{eq:q bound} gives 
\[
\left|\int_{\Omega}u\partial_{i}\psi\, dx\right|\le C\left\Vert F\right\Vert _{(W^{1,r'}(\Omega;\C))'}\left\Vert \psi\right\Vert _{L^{r'}(\Omega;\C)},
\]
 as required. 
\end{proof}
We now are equipped to write the main regularity proposition for $E$,
which will lead to the proof of Theorem~\ref{thm:H1 for E} by a
bootstrap argument. 
\begin{prop}
\label{prop:bootstrap step for E} Assume that \eqref{eq:Hyp-Ellip},
\eqref{eq:Hyp2-epsilon} and \eqref{eq:reg-E} hold. Assume that $E,H\in\Hcurl$
are solutions of \eqref{eq:maxwell} with $G=0$.

Suppose that $E\in L^{q}(\Omega;\C^{3})$ and $H\in L^{s}(\Omega;\C)$,
with $2\leq q,s<\infty$ and write $r=\min((3q+q\delta)(q+3+\delta)^{^{-1}},p,s)$.
Then $E\in W^{1,r}(\Omega;\C^{3})$ and
\begin{multline}
\left\Vert E\right\Vert _{W^{1,r}(\Omega;\C^{3})}\le C\bigl(\left\Vert E\right\Vert _{L^{q}(\Omega;\C^{3})}+\left\Vert H\right\Vert _{L^{s}(\Omega;\C^{3})}+\left\Vert J_{e}\right\Vert _{L^{2}(\Omega)}\\
+\left\Vert J_{m}\right\Vert _{L^{p}(\Omega;\C^{3})}+\left\Vert \div J_{e}\right\Vert _{L^{p}(\Omega;\C)}\bigr),\label{eq:bound norm E in W^1,p}
\end{multline}
 for some constant $C$ depending on $\Omega$, $\Lambda$ given in
\eqref{eq:Hyp-Ellip}, $\k$, $\left\Vert \varepsilon\right\Vert _{W^{1,3+\delta}\left(\Omega;\mathbb{C}^{3\times3}\right)}$
and $r$ only. 
\end{prop}
The corresponding proposition regarding $H$ is as follows. 
\begin{prop}
\label{prop:bootstrap step for H}Assume that \eqref{eq:Hyp-Ellip},
\eqref{eq:Hyp2-mu} and \eqref{eq:reg-H} hold. Assume that $E,H\in\Hcurl$
are solutions of \eqref{eq:maxwell} with $G=0$.

Suppose that $E\in L^{s}(\Omega;\C^{3})$ and $H\in L^{q}(\Omega;\C)$,
with $2\leq q,s<\infty$ and write $r=\min((3q+q\delta)(q+3+\delta)^{^{-1}},p,s)$.
Then $H\in W^{1,r}(\Omega;\C^{3})$ and
\begin{multline*}
\left\Vert H\right\Vert _{W^{1,r}(\Omega;\C^{3})}\le C\bigl(\left\Vert H\right\Vert _{L^{q}(\Omega;\C^{3})}+\left\Vert E\right\Vert _{L^{s}(\Omega;\C^{3})}+\left\Vert J_{m}\right\Vert _{L^{2}(\Omega)}+\left\Vert J_{e}\right\Vert _{L^{p}(\Omega;\C^{3})}\\
+\left\Vert \div J_{m}\right\Vert _{L^{p}(\Omega;\C)}+\left\Vert J_{m}\cdot\nu\right\Vert _{W^{1-\frac{1}{p},p}(\partial\Omega;\C)}\bigr),
\end{multline*}
for some constant $C$ depending on $\Omega$, $\Lambda$ given in
\eqref{eq:Hyp-Ellip}, $\k$, $\left\Vert \mu\right\Vert _{W^{1,3+\delta}\left(\Omega;\mathbb{C}^{3\times3}\right)}$
and $r$ only. 
\end{prop}
We prove both propositions below. We are now ready to prove Theorems~\ref{thm:H1 for E},
\ref{thm:H1 for H} and \ref{thm:W^1,p for E and H}.
\begin{proof}[Proof of Theorems~\ref{thm:H1 for E}, \ref{thm:H1 for H} and \ref{thm:W^1,p for E and H}]
Let us prove Theorem~\ref{thm:H1 for E} first. Considering the
system satisfied by $E-G$ and $H$, we may assume $G=0$. Since $H\in L^{2}\left(\Omega;\mathbb{C}^{3}\right)$,
we may apply Proposition~\ref{prop:bootstrap step for E} with $p=s=2$
a finite number of times with increasing values of $q$. For $q_{n}\geq2$
we obtain $E\in W^{1,r_{n}}\left(\Omega;\mathbb{C}^{3}\right)$, with
$r_{n}=\min(q_{n}(3+\delta)\left(q_{n}+3+\delta\right)^{-1},2)$.
If $r_{n}=2,$ the result is proved. If $r_{n}<2$, Sobolev embeddings
show that $E\in L^{q_{n+1}}\left(\Omega;\mathbb{C}^{3}\right)$ with
\[
q_{n+1}=q_{n}+\frac{\delta q_{n}^{2}}{9+\delta\left(3-q_{n}\right)}\ge q_{n}+\frac{4\delta}{9+\delta},
\]
 using the bounds $q_{n}\ge2$ and $9+\delta\left(3-q_{n}\right)>0$,
which follows from $r_{n}<2$. Thus the sequence $r_{n}$ converges
to $2$ in a finite number of steps. Note that in estimate \eqref{eq:bound norm E in H^1},
$H$ is bounded in terms of $E$ and $J_{m}$ using the simple bound
\[
\Lambda\left|\omega\right|\|H\|_{L^{2}\left(\Omega;\mathbb{C}^{3}\right)}\leq\|\curl E\|_{L^{2}\left(\Omega;\mathbb{C}^{3}\right)}+\|J_{m}\|_{L^{2}\left(\Omega;\mathbb{C}^{3}\right)},
\]
 which follows from \eqref{eq:maxwell}. The proof of Theorem~\ref{thm:H1 for H}
is similar, using Proposition~\ref{prop:bootstrap step for H} in
lieu of Proposition~\ref{prop:bootstrap step for E} to bootstrap.

Let us now turn to Theorem~\ref{thm:W^1,p for E and H}. Suppose
first $p\leq3$ and $\delta<$ 3. From Theorem~\ref{thm:H1 for E}
(resp. Theorem~\ref{thm:H1 for H}) and Sobolev Embeddings, we have
$E\in L^{6}(\Omega;\C^{3})$ (resp. $H\in L^{6}(\Omega;\C^{3})$).
We apply Propositions~\ref{prop:bootstrap step for E} and \ref{prop:bootstrap step for H}
a finite number of times, with $q=s$. Starting with $q_{n}\geq6=q_{0}$
we obtain $E\mbox{ (and }H)\in W^{1,r_{n}}(\Omega;\C^{3})$, with
$r_{n}=\min(q_{n}(3+\delta)(3+\delta+q_{n})^{-1},p)$. If $r_{n}=p$,
the result is proved. If $r_{n}<p$, Sobolev embeddings imply that
$E\mbox{ and }H$ belong to $L^{q_{n+1}}(\Omega;\C^{3})$, with 
\[
q_{n+1}=q_{n}+\frac{\delta q_{n}^{2}}{9+\delta\left(3-q_{n}\right)}\geq q_{n}+\frac{q_{0}^{2}\delta}{9+\delta\left(3-q_{n}\right)}\ge q_{n}+\frac{12\delta}{3-\delta},
\]
 since $q_{n}\ge6$, $\delta<3$ and $9+\delta\left(3-q_{n}\right)>0$
(as $r_{n}<3$). Thus the sequence $r_{n}$ converges to $p$ in a
finite number of steps.

Suppose now $p>3$ and $\delta\in(0,\infty)$. The previous argument
shows that $E$ and $H$ are in $W^{1,3}(\Omega;\C^{3})$. One more
iteration of the argument concludes the proof if $p<3+\delta$, and
shows otherwise that $E$ and $H$ are in $L^{\infty}(\Omega;\C^{3})$,
and the result is obtained by a final application of Propositions~\ref{prop:bootstrap step for E}
and~\ref{prop:bootstrap step for H}. 
\end{proof}
We now prove Proposition~\ref{prop:bootstrap step for E}.
\begin{proof}[Proof of Proposition~\ref{prop:bootstrap step for E}]
We subdivide the proof into four steps.

\emph{Step 1. Variational formulation. }Since $E\times\nu=0$ on $\partial\Omega$,
identity \eqref{eq:very weak E} shows that for every $\varphi\in W^{2,2}(\Omega;\C)$
and $k=1,2,3$ there holds 
\begin{equation}
\int_{\Omega}E_{k}\div(\tepsi\nabla\overline{\varphi})\, dx=\int_{\Omega}F_{k}\cdot\nabla\overline{\varphi}\, dx+\int_{\partial\Omega}(\partial_{k}\overline{\varphi})\varepsilon E\cdot\nu\, d\sigma,\label{eq:eq for E}
\end{equation}
 where we set 
\[
F_{k}=(\partial_{k}\varepsilon)\, E-\varepsilon\left(\mathbf{e}_{k}\times\left(J_{m}-\ii\k\mu H\right)\,\right)-\ik\mathbf{e}_{k}\,\div J_{e}.
\]
 Since $(\partial_{k}\varepsilon)E\in\, L^{q(3+\delta)(q+3+\delta)^{^{-1}}}(\Omega;\C^{3})$,
we have that $F_{k}\in L^{r}(\Omega;\C^{3})$.

\emph{Step 2. Interior regularity. }Given a smooth open subdomain
$\Omega_{0}$ such that $\overline{\Omega_{0}}\subset\Omega$, we
consider a cut-off function $\chi\in C_{0}^{\infty}\left(\Omega;\mathbb{R}\right)$
such that $\chi=1$ in $\Omega_{0}$. A computation gives for $\varphi\in W^{2,2}(\Omega;\C)$
\[
\int_{\Omega}\chi E_{k}\div(\tepsi\nabla\overline{\varphi})\, dx=\int_{\Omega}E_{k}\div(\tepsi\nabla(\chi\overline{\varphi}))\, dx+T_{k}(\varphi),
\]
 where $T_{k}(\varphi)=-\int_{\Omega}E_{k}(\div(\tepsi\overline{\varphi}\nabla\chi)+\varepsilon\nabla\chi\cdot\nabla\overline{\varphi})\, dx$.
Thus, by \eqref{eq:eq for E} we obtain 
\[
\int_{\Omega}\chi E_{k}\div(\tepsi\nabla\overline{\varphi})\, dx=\int_{\Omega}F_{k}\cdot\nabla(\chi\overline{\varphi})\, dx+T_{k}(\varphi),
\]
 since $\chi$ is compactly supported. Using Sobolev embeddings and
the fact that $F_{k}$ is in $L^{r}\left(\Omega;\mathbb{C}^{3}\right)$,
we verify that $\varphi\mapsto\int_{\Omega}F_{k}\cdot\nabla(\chi\overline{\varphi})\, dx+T_{k}(\varphi)$
is in $(W^{1,r'}(\Omega;\C))'$. Thanks to Lemma~\ref{lem:very weak_boundary}
we conclude that $\chi E_{k}\in W^{1,r}(\Omega;\C)$, namely $E\in W^{1,r}(\Omega_{0};\C^{3})$.

\emph{Step 3. Boundary regularity. } Take now $x_{0}\in\partial\Omega$.
Since $\partial\Omega$ is of class $\Cl^{1,1}$ there exists a ball
$B$ centred in $x_{0}$ and an orthogonal change of coordinates $\Phi\in\Cl^{1,1}(B;\R^{3})$
such that in the new system $u_{i}=\Phi_{i}(x)$ we have $\Phi(B\cap\Omega)=\{u_{3}<0\}\cap B(0,R)$.
We can now express the relevant quantities with respect to the coordinates
$u_{1},u_{2},u_{3}$. Let the components of vectors be marked by tildes
if they are expressed in the $u_{i}$ coordinate system. Denoting
$L=\curl E$ we have (see \cite[Lemma 3.1]{WEBER-1981}) 
\[
\tilde{E}=(\nabla\Phi)E,\qquad\tilde{L}=(\nabla\Phi)L=(\partial_{u_{1}},\partial_{u_{2}},\partial_{u_{3}})\times\tilde{E},
\]
 and the corresponding identities for $H$, as $\nabla\Phi$ is an
orthogonal matrix chosen so that $\det\nabla\Phi=1$. Therefore, using
the notation $\tilde{\nabla}=(\partial_{u_{1}},\partial_{u_{2}},\partial_{u_{3}})$,
\eqref{eq:maxwell} implies 
\[
\left\{ \begin{array}{l}
\tilde{\nabla}\times\tilde{H}=\ii\omega\tilde{\varepsilon}\tilde{E}+\tilde{J}_{e},\\
\tilde{\nabla}\times\tilde{E}=-\ii\k\tilde{\mu}\tilde{H}+\tilde{J}_{m},\\
\tilde{E}_{1}=\tilde{E}_{2}=0\quad\mbox{ on }{u_{3}=0}.
\end{array}\right.
\]
 where $\tilde{\varepsilon}=(\nabla\Phi)\epsilon(\nabla\Phi)^{T}$,
$\tilde{J}_{e}=(\nabla\Phi)J_{e}$, $\tilde{\mu}=(\nabla\Phi)\mu(\nabla\Phi)^{T}$
and $\tilde{J}_{m}=(\nabla\Phi)J_{m}$. Namely, Maxwell's equations
\eqref{eq:maxwell} in the new coordinates $u_{i}$ can be written
in the same form. Note that $\tilde{\epsilon}$ and $\tilde{\mu}$
satisfy the ellipticity condition \eqref{eq:Hyp-Ellip} for some $\tilde{\Lambda}>0$.
Moreover, since $\nabla\Phi\in W^{1,\infty}(B;\R^{3\times3})$, the
regularity assumptions \eqref{eq:Hyp2-epsilon} and \eqref{eq:reg-E}
hold for $\tilde{\epsilon}$ and for the sources $\tilde{J}_{e}$,
$\tilde{J}_{m}$. Furthermore, $E\in W^{1,r}(B\cap\Omega;\C^{3})$
if $\tilde{E}\in W^{1,r}(\{u_{3}<0\}\cap B(0,R);\C^{3})$.

We have shown that without loss of generality we can assume that around
$x_{0}$ the boundary is flat. More precisely, suppose that $B\cap\Omega=\{x\cdot\mathbf{e}_{3}<0\}\cap B(0,R)$
and take $\chi\in\mathcal{D}(B;\mathbb{R})$ such that $\chi=1$ in
a neighbourhood $\tilde{B}$ of $x_{0}$.

Let us first consider the two tangential components of $E$, that
is, $E_{j}$ with $j=1,2$. Proceeding as in step 2 we obtain for
every $\varphi\in W^{2,2}(\Omega;\C)\cap W_{0}^{1,2}(\Omega;\C)$
\[
\int_{\Omega}\chi E_{j}\div(\tepsi\nabla\overline{\varphi})\, dx=\int_{\Omega}E_{j}\div(\tepsi\nabla(\chi\overline{\varphi}))\, dx+T_{j}(\varphi),
\]
 where $T_{j}(\varphi)=-\int_{\Omega}E_{j}(\div(\tepsi\overline{\varphi}\nabla\chi)+\varepsilon\nabla\chi\cdot\nabla\overline{\varphi})\, dx$.
In view of identity \eqref{eq:eq for E}, since $\chi\overline{\varphi}=0$
on $\partial\Omega$, we have 
\[
\int_{\Omega}\chi E_{j}\div(\tepsi\nabla\overline{\varphi})\, dx=\int_{\Omega}F_{j}\cdot\nabla(\chi\overline{\varphi})\, dx+T_{j}(\varphi).
\]
 As in step 2, thanks to Lemma~\ref{lem:very weak_boundary} we conclude
that $\chi E_{j}\in W^{1,r}(\Omega;\C)$ for $j=1,2$.

Let us now turn to the normal component $E_{3}$. Consider the second
part of Maxwell's equations \eqref{eq:maxwell}, $\curl E=-\ii\k{\mu}H+J_{m}$
in the quotient space where every element of $L^{r}\bigl(\tilde{B};\C^{3}\bigr)$
is identified with nought, that is, $W^{-1,2}\bigl(\tilde{B};\C^{3}\bigr)/L{}^{r}\bigl(\tilde{B};\C^{3}\bigr)$.
We find $0=-\curl E=-\curl\bigl(E_{3}\mathbf{e}_{3}\bigr)=\mathbf{e}_{3}\times\nabla{E}_{3},$
since $-\ii\k{\mu}H+J_{m}\in L^{r}\left(\Omega;\mathbb{C}^{3}\right)$
and $E_{1},E_{2}\in W^{1,r}(\tilde{B};\C)$. In other words, 
\begin{equation}
\nabla{E}_{3}=\mathbf{e}_{3}\left(\mathbf{e}_{3}\cdot\nabla{E}_{3}\right)\mbox{ in }W^{-1,2}\bigl(\tilde{B};\C^{3}\bigr)/L{}^{r}\bigl(\tilde{B};\C^{3}\bigr).\label{eq:tildeE3-2}
\end{equation}
 Therefore, taking now the divergence of the first identity in Maxwell's
equations \eqref{eq:maxwell}, and using the fact that $\div J_{e}\in L^{r}\left(\Omega;\mathbb{C}\right)$
and $E\in L^{r}\left(\Omega;\mathbb{C}^{3}\right)$ we obtain, in
the quotient space $W^{-1,2}\bigl(\tilde{B};\C\bigr)/L{}^{r}\bigl(\tilde{B};\C\bigr)$,
\[
0=\div(\varepsilon E)=\div(E_{3}\epsilon\mathbf{e}_{3})=\nabla E_{3}\cdot(\epsilon\mathbf{e}_{3})=(\mathbf{e}_{3}\cdot\nabla E_{3})\,\mathbf{e}_{3}\cdot(\epsilon\mathbf{e}_{3}).
\]
 Hence, in view of \eqref{eq:Hyp-Ellip} and \eqref{eq:tildeE3-2}
we obtain $\nabla{E}{}_{3}\in L{}^{r}\bigl(\tilde{B};\C^{3}\bigr)$,
and therefore $E\in W^{1,r}\bigl(\tilde{B};\C^{3}\bigr)$.

\emph{Step 4. Global regularity.} Combining the interior and the boundary
regularities, a standard ball covering argument shows $E\in W^{1,r}(\Omega;\C^{3})$.
The estimate given in \eqref{eq:bound norm E in W^1,p} follows from
Lemma~\ref{lem:very weak_boundary}. \qquad{}\endproof
\end{proof}
We now turn to the proof of Proposition~\ref{prop:bootstrap step for H}.
Naturally, the interior estimates can be obtained in the exact same
way, substituting the very weak formulations for the components of
$E$ by the corresponding identities for the components of $H$. The
boundary estimates require different arguments, and we detail this
step below.
\begin{proof}[Proof of Proposition~\ref{prop:bootstrap step for H}. Boundary regularity]
 First note that it is sufficient to consider the case when $J_{m}\cdot\nu=0$
on $\partial\Omega$.

Indeed, as we assumed that $J_{m}\cdot\nu\in W^{1-\frac{1}{p},p}(\partial\Omega;\C)$,
there exists $j_{m}\in W^{1,p}(\Omega;\C)$ such that $j_{m}=J_{m}\cdot\nu$
in the sense of traces on $\partial\Omega$. Since $\partial\Omega$
is of class $\Cl^{1,1}$, there exists $h\in\Cl^{0,1}\left(\overline{\Omega},\mathbb{R}^{3}\right)$
such that $h=\nu$ on $\partial\Omega$. Now, notice that $(E',H')=(E,H+i\omega^{-1}\mu^{-1}j_{m}h)$
are solutions of a Maxwell's system of equations with the same boundary
condition, but with the currents being changed to $J'_{m}=J_{m}-j_{m}h$
and $J'_{e}=J_{e}+i\omega^{-1}\curl(j_{m}\mu^{-1}h)$. We have $J'_{m}\in W^{1,p}(\mbox{div},\Omega)$
and $J'_{m}\cdot\nu=0$ on $\partial\Omega$, whereas $J'_{e}\in L^{\tilde{p}}(\Omega;\C^{3})$,
with $\tilde{p}=\min(p,3+\delta)$. Since $i\omega^{-1}\mu^{-1}j_{m}h\in W^{1,r}(\Omega;\C^{3})$
the regularity of $H'$ will imply that of $H$.

We next observe that, as $E\times\nu=0$ on $\partial\Omega$, we
may write 
\begin{equation}
0=\div_{\partial\Omega}(E\times\nu)=(\curl E)\cdot\nu=-\ii\k{\mu}H\cdot\nu+J_{m}\cdot\nu\mbox{ in }H^{-\frac{1}{2}}(\partial\Omega),\label{eq:monk}
\end{equation}
 see e.g. \cite[(3.52)]{MONK-2003}. In other words, ${\mu}H\cdot\nu=0$
in $H^{-\frac{1}{2}}(\partial\Omega)$. Then, by \eqref{eq:very weak H},
for every $\varphi\in W^{2,2}(\Omega;\C)$ and $k=1,2,3$ there holds
\begin{equation}
\int_{\Omega}H_{k}\div(\tmuu\nabla\overline{\varphi})\, dx=\int_{\Omega}G_{k}\cdot\nabla\overline{\varphi}\, dx-\int_{\partial\Omega}\left(\mathbf{e}_{k}\times\left(H\times\nu\right)\right)\cdot({\mu}^{T}\nabla\overline{\varphi})\, d\sigma,\label{eq:eq for H}
\end{equation}
 where 
\[
G_{k}=(\partial_{k}{\mu})H-{\mu}\left(\mathbf{e}_{k}\times\left(J_{e}+\ii\k\varepsilon E\right)\right)+\ik\mathbf{e}_{k}\div J_{m}\in L^{r}(\Omega;\C^{3}),
\]
 since $(\partial_{k}{\mu})H\in L^{q(3+\delta)(q+3+\delta)^{^{-1}}}(\Omega;\C^{3})$.

Take $x_{0}\in\partial\Omega$. As in the proof of Proposition~\ref{prop:bootstrap step for E},
we can assume that $\partial\Omega$ is the plane $x\cdot\mathbf{e}_{3}=0$
in a neighbourhood $B$ of $x_{0}$. Again let us focus on the tangential
components first. Take $\chi\in\mathcal{D}(B;\mathbb{R})$ such that
$\chi=1$ in a neighbourhood $\tilde{B}$ of $x_{0}$ and $j=1,2$.

We choose a test function satisfying a Neumann type boundary condition,
that is $\varphi\in W^{2,2}(\Omega;\C)$ such that $\tmuu\nabla\overline{\varphi}\cdot\nu=0$
on $\partial\Omega$. We have 
\[
\int_{\Omega}\chi{H}{}_{j}\div(\tmuu\nabla\overline{\varphi})\, dx=\int_{\Omega}H_{j}\div(\tmuu\nabla(\chi\overline{\varphi}))\, dx+R(\varphi),
\]
 where $R(\varphi)=-\int_{\Omega}H_{j}\bigl(\div(\tmuu\overline{\varphi}\nabla\chi)+{\mu}\nabla\chi\cdot\nabla\overline{\varphi}\bigr)\, dx$.
From identity \eqref{eq:eq for H} we obtain 
\begin{equation}
\int_{\Omega}\chi{H}{}_{j}\div(\tmuu\nabla\overline{\varphi})\, dx=\int_{\Omega}G_{j}\cdot\nabla(\chi\overline{\varphi})\, dx+R(\varphi)+S(\varphi),\label{eq:weak for H'}
\end{equation}
 where 
\[
S(\varphi)=-\int_{\partial\Omega}\left(\mathbf{e}_{j}\times\left(H\times\mathbf{e}_{3}\right)\right)\cdot({\mu}^{T}\nabla(\chi\overline{\varphi}))\, d\sigma.
\]
 As before, the functional $\varphi\mapsto\int_{\Omega}G_{j}\cdot\nabla(\chi\overline{\varphi})\, dx+R(\varphi)$
is in $(W^{1,r'}(\Omega;\C))'$. We shall now prove that $S\in(W^{1,r'}(\Omega;\C))'$.
Since $\tmuu\nabla\overline{\varphi}\cdot\nu=0$ on $\partial\Omega$
and $\nu=\mathbf{e}_{3}$ on $B$, we have $\chi\left(\mathbf{e}_{j}\times\left(H\times\mathbf{e}_{3}\right)\right)\cdot({\mu}^{T}\nabla\overline{\varphi})=0$,
thus 
\[
S(\varphi)=-\int_{\partial\Omega}\left(\mathbf{e}_{j}\times\left(H\times\mathbf{e}_{3}\right)\right)\cdot({\mu}^{T}\nabla\chi)\,\overline{\varphi}\, d\sigma.
\]
 By hypothesis we have $H\in W^{1,r}(\curl,\Omega)$, whence $H\times\nu\in W^{-1/r,r}(\partial\Omega;\C^{3})$.
It follows that $\left(\mathbf{e}_{j}\times\left(H\times\mathbf{e}_{3}\right)\right)\cdot({\mu}^{T}\nabla\chi)\in W^{-1/r,r}(\partial\Omega;\C^{3})$.
As a result (see \cite[Theorem 1.5.1.2]{GRISVARD-1985}), 
\[
\left|S(\varphi)\right|\le C\left\Vert \phi\right\Vert _{W^{1-\frac{1}{r'},r'}(\partial\Omega;\C)}\le C\left\Vert \phi\right\Vert _{W^{1,r'}(\Omega;\C)},
\]
 for some $C>0$ independent of $\varphi$; in other words $S\in(W^{1,r'}(\Omega;\C))'$.
We can now apply Lemma~\ref{lem:very weak_boundary} to \eqref{eq:weak for H'}
and obtain $\chi{H}\cdot\mathbf{e}_{j}\in W^{1,r}(\Omega;\C)$. The
rest of the proof follows faithfully that of Proposition~\ref{prop:bootstrap step for E}.
\end{proof}
To conclude this section, we point out that higher regularity results
follow naturally under appropriate assumptions.
\begin{thm}
\label{thm:higher regularity-new} Suppose that \eqref{eq:Hyp-Ellip}
holds and take $N\in\mathbb{N}^{*}$. Assume additionally that $\partial\Omega$
is of class $\Cl^{N,1}$ and that 
\begin{eqnarray*}
 & \varepsilon, & {\mu}\in W^{N,p}\left(\Omega;\mathbb{C}^{3\times3}\right),\qquad J_{e},J_{m}\in W^{N,p}\left(\div,\Omega\right),\\
 & J_{m} & \cdot\nu\in W^{N-\frac{1}{p},p}(\partial\Omega;\C),\qquad G\in W^{N,p}(\Omega;\C^{3}),
\end{eqnarray*}
 for some $p>3$. If $E$ and $H$ in $H\left(\curl,\Omega\right)$
are weak solutions of \eqref{eq:maxwell}, then $E,H\in W^{N,p}\left(\Omega;\mathbb{C}^{3}\right)$
and there holds
\begin{multline*}
\left\Vert E\right\Vert _{W^{N,p}\left(\Omega;\mathbb{C}^{3}\right)}+\left\Vert H\right\Vert _{W^{N,p}\left(\Omega;\mathbb{C}^{3}\right)}\le C\left(\left\Vert E\right\Vert _{L^{2}\left(\Omega\right)}+\left\Vert H\right\Vert _{L^{2}\left(\Omega\right)}+\left\Vert G\right\Vert _{W^{N,p}\left(\Omega;\C^{3}\right)}\right.\\
+\Bigl.\left\Vert J_{e}\right\Vert _{W^{N,p}\left(\div,\Omega\right)}+\left\Vert J_{m}\right\Vert _{W^{N,p}\left(\div,\Omega\right)}+\left\Vert J_{m}\cdot\nu\right\Vert _{W^{N-\frac{1}{p},p}(\partial\Omega;\C)}\Bigr),
\end{multline*}
for some constant $C$ depending on $\Omega$, $\Lambda$, $\k$,
$\left\Vert \varepsilon\right\Vert _{W^{N,p}\left(\Omega;\mathbb{C}^{3\times3}\right)}$
and $\left\Vert \mu\right\Vert _{W^{N,p}\left(\Omega;\mathbb{C}^{3\times3}\right)}$. \end{thm}
\begin{proof}
The proof is done by induction. Theorem~\ref{thm:W^1,p for E and H}
corresponds to $N=1$. Assume that for some $N\geq2$, Theorem~\ref{thm:higher regularity-new}
holds for $N-1$.

For simplicity, we shall consider \eqref{eq:maxwell} in its strong
form, but every step can be made rigorous by passing to the suitable
weak formulation.

By using a change of coordinates as in the proof of Proposition \ref{prop:bootstrap step for E},
we can assume without loss of generality that $\Omega\cap B(0,R)=\{x\cdot\mathbf{e}_{3}<0\}\cap B(0,R)$.
Indeed, the assumption $\partial\Omega\in\Cl^{N,1}$ implies that
the regularity assumptions on the coefficients and on the source terms
and the conditions $E,H\in W^{N,p}$ are insensitive to a $\Cl^{N,1}$
change of coordinates.

For $i=1,2$ we have 
\[
\left\{ \begin{array}{l}
\curl\,\partial_{i}H=\ii\k\varepsilon\partial_{i}E+J_{e}'\qquad\mbox{ in }\Omega\cap B(0,R),\\
\curl\,\partial_{i}E=-\ii\k{\mu}\partial_{i}H+J_{m}'\qquad\mbox{ in }\Omega\cap B(0,R),\\
\partial_{i}E\times\mathbf{e}_{3}=\partial_{i}G\times\mathbf{e}_{3}\mbox{ on }\partial\Omega\cap\{x\cdot\mathbf{e}_{3}=0\}\cap\bar{B}(0,R),
\end{array}\right.
\]
 where $J_{e}'=J_{e}+\ii\k(\partial_{i}\varepsilon)E$ and $J_{m}'=J_{m}-\ii\k(\partial_{i}\mu)H$.
By assumption, we have $E,H\in W^{N-1,p}\left(\Omega;\mathbb{C}^{3}\right)$,
therefore $J_{e}',J_{m}'\in W^{N-1,p}\left(\div,\Omega\right)$ and
$J_{m}'\cdot\nu\in W^{N-1-\frac{1}{p},p}(\partial\Omega;\C)$. Applying
Theorem~\ref{thm:higher regularity-new} with $N-1$ in lieu of $N$
to the above system shows that $\partial_{i}E,\partial_{i}H\in W^{N-1,p}\left(\Omega;\mathbb{C}^{3}\right)$.

An argument similar to the one given in the third step of the proof
of Proposition~\ref{prop:bootstrap step for E} allows us to infer
that $\partial_{3}E,\partial_{3}H\in W^{N-1,p}\left(\Omega;\mathbb{C}^{3}\right)$;
as a consequence, $E,H\in W^{N,p}\left(\Omega;\mathbb{C}^{3}\right)$.
The corresponding norm estimate follows by Theorem~\ref{thm:W^1,p for E and H}
and the argument given above. 
\end{proof}

\section{\label{sec:Bi-anisotropic} Bi-anisotropic materials}

In this section, we investigate the interior regularity of the solutions
of the following problem 
\begin{equation}
\left\{ \begin{array}{l}
\curl H=\ii\k\left(\varepsilon E+\xi H\right)+J_{e}\qquad\mbox{ in }\Omega,\\
\curl E=-\ii\k\left(\zeta E+\mu H\right)+J_{m}\qquad\mbox{ in }\Omega.
\end{array}\right.\label{eq:maxwell for bi-anisotropic-1}
\end{equation}
 As far as the authors are aware, this question was previously studied
only recently in \cite{FERNANDES-2012}, where the parameters are
assumed to be at least Lipschitz continuous. In this more general
context, hypothesis \eqref{eq:Hyp-Ellip} is not sufficient to ensure
ellipticity. As we will see in Proposition~\ref{prop:coupled_system-bianosotropic},
the leading order parameter for the coupled elliptic system is the
tensor 
\begin{equation}
A=A_{ij}^{\alpha\beta}=\left[\begin{array}{cccc}
\Re\varepsilon & -\Im\varepsilon & \Re\xi & -\Im\xi\\
\Im\varepsilon & \Re\varepsilon & \Im\xi & \Re\xi\\
\Re\zeta & -\Im\zeta & \Re{\mu} & -\Im{\mu}\\
\Im\zeta & \Re\zeta & \Im{\mu} & \Re{\mu}
\end{array}\right],\label{eq:matrix A}
\end{equation}
where the Latin indices $i,j=1,\dots,4$ identify the different $3\times3$
block sub-matrices, whereas the Greek letters $\alpha,\beta=1,2,3$
span each of these $3\times3$ block sub-matrices. We assume that
$A$ is in $L^{\infty}(\Omega;\mathbb{R})^{12\times12}$ and satisfies
a strong Legendre condition (as in \cite{CHEN-WU-1998,GIAQUINTA-MARTINAZZI-2005}),
that is, there exists $\Lambda>0$ such that 
\begin{equation}
A_{ij}^{\alpha\beta}\ve_{\alpha}^{i}\ve_{\beta}^{j}\ge\Lambda\left|\ve\right|^{2},\;\ve\in\mathbb{R}^{12}\quad\mbox{ and}\quad\bigl|A_{ij}^{\alpha\beta}\bigr|\le\Lambda^{-1}\qquad\mbox{ a.e. in \ensuremath{\Omega}.}\label{eq:strong legendre condition}
\end{equation}
 The following result gives a sufficient condition for \eqref{eq:strong legendre condition}
to hold true. 
\begin{lem}
Assume that $\varepsilon_{0},{\mu}_{0},\kappa,\chi$ are real constants,
with $\varepsilon_{0}>0$ and $\mu_{0}>0$. Let 
\begin{equation}
\varepsilon=\varepsilon_{0}I_{3},\quad{\mu}={\mu}_{0}I_{3},\quad\xi=(\chi-\ii\kappa)I_{3},\quad\zeta=(\chi+\ii\kappa)I_{3},\label{eq:constitutive relations for chiral materials}
\end{equation}
 where $I_{3}$ is the $3\times3$ identity matrix, and construct
the matrix $A$ as in \eqref{eq:matrix A}. If 
\begin{equation}
\chi^{2}+\kappa^{2}<\varepsilon_{0}{\mu}_{0},\label{eq:chiral materials}
\end{equation}
 then $A$ satisfies \eqref{eq:strong legendre condition}. \end{lem}
\begin{rem}
This result shows that a wide class of materials satisfy the strong
Legendre condition \eqref{eq:strong legendre condition}. Considering
for simplicity the case of constant and isotropic parameters, the
constitutive relations given in \eqref{eq:constitutive relations for chiral materials}
describe the so-called chiral materials. It turns out that \eqref{eq:chiral materials}
is satisfied for natural materials \cite{WANG-ET-AL-2009}. \end{rem}
\begin{proof}
A direct calculation shows that the smallest eigenvalue of $A$ is
$(\varepsilon_{0}+{\mu}_{0}-(\varepsilon_{0}^{2}-2\varepsilon_{0}{\mu}_{0}+{\mu}_{0}^{2}+4\chi^{2}+4\kappa^{2})^{1/2})/2$,
which is strictly positive since $\chi^{2}+\kappa^{2}<\varepsilon_{0}{\mu}_{0}$.
\end{proof}
We now give the regularity assumptions on the parameters. In contrast
to the previous situation, here the mixing coefficients $\xi$ and
$\zeta$ fully couple electric and magnetic properties. We are thus
led to assume that 
\begin{equation}
\varepsilon,\xi,\mu,\zeta\in W^{1,3+\delta}\left(\Omega;\mathbb{C}^{3\times3}\right)\mbox{ for some }\delta>0.\label{eq:Hyp2-epsilon-xi-mu-zeta}
\end{equation}
 The theorem below shows that at least as far as interior regularity
is concerned, Theorem~\ref{thm:W^1,p for E and H} also applies in
this more general setting. 
\begin{thm}
\label{thm:C 0 alpha -bi-anisotropic} Assume that \eqref{eq:strong legendre condition}
and \eqref{eq:Hyp2-epsilon-xi-mu-zeta} hold. Suppose that the current
sources $J_{e}$ and $J_{m}$ are in $W^{1,p}(\div,\Omega)$ for some
$p\geq2$.

If $E$ and $H$ in $\Hcurl$ are weak solutions of \eqref{eq:maxwell for bi-anisotropic-1},
then $E,H\in W_{loc}^{1,q}(\Omega;\C^{3})$ with $q=\min(p,3+\delta)$.
Furthermore, for any open set $\Omega_{0}$ such that $\overline{\Omega_{0}}\subset\Omega$
there holds 
\[
\left\Vert E\right\Vert _{W^{1,q}\left(\Omega_{0};\C^{3}\right)}+\left\Vert H\right\Vert _{W^{1,q}\left(\Omega_{0};\C^{3}\right)}\le C\bigl(\left\Vert E\right\Vert _{L^{2}\left(\Omega\right)}+\left\Vert H\right\Vert _{L^{2}\left(\Omega\right)}+\left\Vert (J_{e},J_{m})\right\Vert _{W^{1,p}(\div,\Omega)^{2}}\bigr),
\]
 where $C$ is a constant depending on $\Omega,\Omega_{0},q,\Lambda,\k,$
and the $W^{1,3+\delta}\left(\Omega;\mathbb{C}^{3\times3}\right)$
norms of $\varepsilon$ \textup{, $\mu$, $\xi$ and $\zeta$.} In
particular, if $p>3$ then $E,H\in\Cl_{loc}^{0,1-\frac{3}{q}}(\Omega;\C^{3})$. 
\end{thm}
We did not investigate the regularity up to the boundary in this problem
for two reasons. From a modeling point of view, the natural (or pertinent)
boundary conditions to be considered in this case are not completely
clear. There is also a technical reason: the boundary terms are a
rather intricate mix of Neumann and Dirichlet type terms on both $E$
and $H$, and the correct space of test functions to consider is not
readily apparent (see Proposition~\ref{prop:coupled_system-bianosotropic}).

The proof of this result is a variant of the proof of Theorem~\ref{thm:W^1,p for E and H}.
In this case, the system is written in $\mathbb{R}^{12}$ (instead
of a weakly coupled system of $6$ complex unknowns) and the proof
is detailed in the appendix.

\section{\label{sec:Campanato} Proof of Theorem~\ref{thm:global campanato for E}
using Campanato estimates}

The purpose of this section is to prove Theorem~\ref{thm:global campanato for E}.
We shall apply classical Campanato estimates for elliptic equations
to \eqref{eq:a}, namely the elliptic equations satisfied by $E$.
We first state the properties of Campanato spaces that we shall use,
and then proceed to the proof of Theorem~\ref{thm:global campanato for E}.

For $\lambda\ge0$ and $p\ge1$ we denote the Campanato space by $L^{p,\lambda}(\Omega;\C)$
\cite{CAMPANATO-1980}, namely the Banach space of functions $u\in L^{p}\left(\Omega;\C\right)$
such that 
\[
[u]_{p,\lambda;\Omega}^{p}:=\sup_{x\in\Omega,0<\rho<{\rm diam}\Omega}\rho^{-\lambda}\int_{\Omega(x,\rho)}\left|u(y)-\frac{1}{|\Omega(x,\rho)|}\int_{\Omega(x,\rho)}u(z)\, dz\right|^{p}\, dy<\infty,
\]
 where $\Omega(x,\rho)=\Omega\cap\{y\in\mathbb{R}^{3}:\left|y-x\right|<\rho\}$,
equipped with the norm 
\[
\left\Vert u\right\Vert _{L^{p,\lambda}(\Omega;\C)}=\left\Vert u\right\Vert _{L^{p}(\Omega;\C)}+[u]_{p,\lambda;\Omega}.
\]

\begin{lem}
\label{lem:campanato_properties}Take $\lambda\ge0$. 
\begin{enumerate}
\item Suppose $\lambda>3$. If $u\in L^{2,\lambda}\left(\Omega;\C\right)$
then $u\in\Cl^{0,\frac{\lambda-3}{2}}\left(\overline{\Omega};\C\right)$,
and the embedding is continuous. 
\item Suppose $\lambda<3$. If $u\in L^{2}(\Omega;\C)$ and $\nabla u\in L^{2,\lambda}\left(\Omega;\C^{3}\right)$
then $u\in L^{2,2+\lambda}\left(\Omega;\C\right)$, and the embedding
is continuous. 
\item Suppose $\delta>0$ and $\lambda\neq1$. If $f\in L^{3+\delta}\left(\Omega;\C\right)$
and $u\in L^{2}\left(\Omega;\C\right)$ with $\nabla u\in L^{2,\lambda}\left(\Omega;\C^{3}\right)$
then $fu\in L^{2,\lambda'}\left(\Omega;\C\right)$ with $\lambda'=\min(\lambda+2\delta(3+\delta)^{-1},3(1+\delta)(3+\delta)^{-1})$,
and the embedding is continuous. 
\end{enumerate}
\end{lem}
\begin{proof}
Statements (1) and (2) are classical, see e.g. \cite[Chapter 1]{TROIANIELLO-1987}.
For (3), note that Hölder's inequality implies that $f\in L^{2,3(1+\delta)(3+\delta)^{-1}}(\Omega;\C)$.
When $\lambda<1$, the result follows from \cite[Lemma 4.1]{DIFAZIO-1993}.
When $\lambda>1$, (3) follows from (1) and (2).
\end{proof}
We now state the regularity result regarding Campanato estimates we
will use. It can be found in \cite{TROIANIELLO-1987} (Theorems 2.19
and 3.16). 
\begin{prop}
\label{prop:campanato_regularity} Assume \eqref{eq:Hyp-Ellip} and
$\Im{\mu}=0$. There exists $\lambda_{\mu}\in(1,2]$ depending only
on $\Omega$ and on $\Lambda$ given in \eqref{eq:Hyp-Ellip}, such
that if $F\in L^{2,\lambda}\left(\Omega;\mathbb{C}^{3}\right)$ for
some $\lambda\in[0,\lambda_{{\mu}})$, and $u\in W^{1,2}\left(\Omega;\mathbb{C}\right)$
satisfies 
\[
\left\{ \begin{array}{l}
-\div({\mu}\nabla u)=\div(F)\qquad\mbox{ in }\Omega,\\
{\mu}\nabla u\cdot\nu=F\cdot\nu\qquad\mbox{ on }\partial\Omega,
\end{array}\right.
\]
 then $\nabla u\in L^{2,\lambda}\left(\Omega;\mathbb{C}^{3}\right)$
and 
\begin{equation}
\left\Vert \nabla u\right\Vert _{L^{2,\lambda}\left(\Omega;\mathbb{C}^{3}\right)}\le C\left\Vert F\right\Vert _{L^{2,\lambda}\left(\Omega;\mathbb{C}^{3}\right)},\label{eq:estimate campanato}
\end{equation}
 where the constant $C$ depends only on $\Lambda,\lambda$ and $\Omega$.

Alternatively, assume \eqref{eq:Hyp-Ellip} and \eqref{eq:Hyp2-mu}.
For all $\lambda\in[0,2]$, if $F\in L^{2,\lambda}\left(\Omega;\mathbb{C}^{3}\right)$,
$f\in L^{2}\left(\Omega;\mathbb{C}\right)$, and $u\in W^{1,2}\left(\Omega;\mathbb{C}\right)$
satisfies 
\[
\left\{ \begin{array}{l}
-\div({\mu}\nabla u)=\div(F)+f\qquad\mbox{ in }\Omega,\\
u=0\qquad\mbox{ on }\partial\Omega,
\end{array}\right.
\]
 then $\nabla u\in L^{2,\lambda}\left(\Omega;\mathbb{C}^{3}\right)$
and 
\begin{equation}
\left\Vert \nabla u\right\Vert _{L^{2,\lambda}\left(\Omega;\mathbb{C}^{3}\right)}\le C\left(\left\Vert F\right\Vert _{L^{2,\lambda}\left(\Omega;\mathbb{C}^{3}\right)}+\left\Vert f\right\Vert _{L^{2}\left(\Omega;\mathbb{C}\right)}\right),\label{eq:estimate campanato-1}
\end{equation}
 where the constant $C$ depends on $\Lambda$, $\Omega$, and $\left\Vert \mu\right\Vert _{W^{1,3+\delta}\left(\Omega;\mathbb{C}^{3\times3}\right)}$
only. 
\end{prop}
We first study the regularity of $H$ following a variant of an argument
given in \cite{YIN-2004}. 
\begin{prop}
\label{prop:H campanato}Assume that $\Omega$ is simply connected,
that \eqref{eq:Hyp-Ellip} holds with $\Im{\mu}=0$ and $J_{m}\in L^{2,\lambda}(\Omega)$
with $1<\lambda<\lambda_{{\mu}}$, where $\lambda_{{\mu}}$ is given
by Proposition~\ref{prop:campanato_regularity}. Let $E$ and $H$
in $\Hcurl$ be weak solutions of \eqref{eq:maxwell} with $G=0$.
Then $H\in L^{2,\lambda}\left(\Omega;\mathbb{C}^{3}\right)$ and 
\begin{equation}
\left\Vert H\right\Vert _{L^{2,\lambda}\left(\Omega;\mathbb{C}^{3}\right)}\le C\bigl(\left\Vert E\right\Vert _{L^{2}(\Omega)}+\left\Vert J_{e}\right\Vert _{L^{2}(\Omega)}+\left\Vert J_{m}\right\Vert _{L^{2,\lambda}(\Omega;\C^{3})}\bigr),\label{eq:estimate campanato H}
\end{equation}
 where the constant $C$ depends only on $\Lambda$, $\lambda$, $\k$
and $\Omega$. \end{prop}
\begin{proof}
Since $\ii\k\varepsilon E+J_{e}$ is divergence free in $\Omega$,
and $\Omega$ is $\Cl^{1,1}$ and simply connected, it is well known
that there exists $T\in H^{1}(\Omega)$ such that $\ii\k\varepsilon E+J_{e}=\curl T$,
satisfying 
\begin{equation}
\left\Vert T\right\Vert _{H^{1}\left(\Omega\right)}\leq C\bigl(\left\Vert J_{e}\right\Vert _{L^{2}(\Omega)}+\left\Vert E\right\Vert _{L^{2}(\Omega)}\bigr)\label{eq:campa1}
\end{equation}
 where $C$ depends on $\Omega$, $\Lambda$ given in \eqref{eq:Hyp-Ellip}
and $\omega$ only, see e.g. \cite[Chapter I, Theorem 3.5]{GIRAULT-RAVIART-1986}.
Thanks to Lemma~\ref{lem:campanato_properties}, this implies $\mu T\in L^{2,2}(\Omega;\C^{3})\subset L^{2,\lambda}(\Omega;\C^{3})$,
and therefore $\mu T+i\k^{-1}J_{m}\in L^{2,\lambda}(\Omega;\C^{3})$.

As $H-T$ is curl free in $\Omega$, in view of \cite[Chapter I, Theorem 2.9]{GIRAULT-RAVIART-1986}
there exists $h\in H^{1}(\Omega;\C)$ such that $H-T=\nabla h$. The
potential $h$ is defined up to a constant by 
\[
\left\{ \begin{array}{l}
\div({\mu}\nabla h)=\div(-\mu T-i\k^{-1}J_{m})\qquad\mbox{ in }\Omega,\\
{\mu}\nabla h\cdot\nu=(-\mu T-i\k^{-1}J_{m})\cdot\nu\qquad\mbox{ on }\partial\Omega.
\end{array}\right.
\]
 Note that the boundary condition follows from that of $E$ and \eqref{eq:monk}.
Thanks to estimate \eqref{eq:estimate campanato} in Proposition~\ref{prop:campanato_regularity},
we have 
\begin{equation}
\left\Vert \nabla h\right\Vert _{L^{2,\lambda}\left(\Omega;\mathbb{C}^{3}\right)}\le C\left\Vert \mu T+i\k^{-1}J_{m}\right\Vert _{L^{2,\lambda}\left(\Omega;\mathbb{C}^{3}\right)}\leq\tilde{C}\bigl(\left\Vert T\right\Vert _{H^{1}\left(\Omega\right)}+\left\Vert J_{m}\right\Vert _{L^{2,\lambda}\left(\Omega;\mathbb{C}^{3}\right)}\bigr).\!\!\!\!\!\!\!\!\label{eq:campa2}
\end{equation}
 The conclusion follows from the identity $H=T+\nabla h$ and the
estimates \eqref{eq:campa1} and \eqref{eq:campa2}. 
\end{proof}
We now adapt Proposition~\ref{prop:bootstrap step for E} to be able
to use Campanato estimates in the bootstrap argument. 
\begin{prop}
\label{prop:bootstrap step for E-campanato}Assume that $\Omega$
is simply connected, that \eqref{eq:Hyp-Ellip} holds with $\Im{\mu}=0$
and that \eqref{eq:Hyp2-epsilon} holds. Suppose $J_{m}\in L^{2,\tilde{\lambda}}(\Omega;\C^{3})$
and $\div J_{e}\in L^{2,\tilde{\lambda}}(\Omega;\C)$ for some $1<\tilde{\lambda}<\lambda_{{\mu}}$,
where $\lambda_{{\mu}}$ is given by Proposition~\ref{prop:campanato_regularity}.
Let $E$ and $H$ in $\Hcurl$ be weak solutions of \eqref{eq:maxwell}
with $G=0$.

If $\nabla E\in L^{2,\lambda_{0}}\left(\Omega;\mathbb{C}^{3\times3}\right)$
for some $\lambda_{0}\in[0,\infty)\setminus\{1\}$ then $\nabla E\in L^{2,\lambda_{1}}(\Omega;\C)^{9}$,
with $\lambda_{1}=\min\left(\tilde{\lambda},\lambda_{0}+2\delta(3+\delta)^{-1},3(1+\delta)(3+\delta)^{-1}\right)$.
Moreover there holds
\begin{multline}
\left\Vert \nabla E\right\Vert _{L^{2,\lambda_{1}}(\Omega;\C)^{9}}\le C\left(\left\Vert E\right\Vert _{L^{2}(\Omega)}+\left\Vert \nabla E\right\Vert _{L^{2,\lambda_{0}}(\Omega;\C)^{9}}+\left\Vert J_{e}\right\Vert _{L^{2}(\Omega)}\right.\\
+\left.\left\Vert J_{m}\right\Vert _{L^{2,\tilde{\lambda}}(\Omega;\C^{3})}+\left\Vert \div J_{e}\right\Vert _{L^{2,\tilde{\lambda}}(\Omega;\C)}\right),\label{eq:bootstrap campanato}
\end{multline}
where the constant $C$ depends only on $\Omega$, $\Lambda$, $\lambda_{1}$,
$\k$ and $\left\Vert \varepsilon\right\Vert _{W^{1,3+\delta}\left(\Omega;\mathbb{C}^{3\times3}\right)}$. \end{prop}
\begin{proof}
In view of Theorem~\ref{thm:H1 for E} and Proposition~\ref{prop:coupled_system},
for each $k=1,2,3$, $E_{k}\in H^{1}(\Omega;\C)$ is a weak solution
of 
\begin{equation}
-\div\left(\varepsilon\nabla E_{k}\right)=\div\left(\partial_{k}\varepsilon\, E+S_{k}\right)\textrm{ in }\Omega,\label{eq:campa-E1}
\end{equation}
 with 
\[
S_{k}=-\varepsilon\left(\mathbf{e}_{k}\times\left(J_{m}-\ii\k\mu H\right)\right)-\ik\mathbf{e}_{k}\div J_{e}.
\]
 Thanks to Proposition~\ref{prop:H campanato} we have that $S_{k}\in L^{2,\tilde{\lambda}}(\Omega;\C)$.
Furthermore there holds $\partial_{k}\varepsilon E\in L^{2,\tilde{\lambda}_{0}}\left(\Omega;\mathbb{C}^{3}\right)$
with $\tilde{\lambda}_{0}=\min\left(\lambda_{0}+2\delta(3+\delta)^{-1},3(1+\delta)(3+\delta)^{-1}\right)$
in view of Lemma~\ref{lem:campanato_properties}. Thus 
\begin{equation}
\partial_{k}\varepsilon\, E+S_{k}\in L^{2,\lambda_{1}}\left(\Omega;\mathbb{C}^{3}\right).\label{eq:rhs-est-campa}
\end{equation}

\emph{Interior regularity.} Given a smooth open subdomain $\Omega_{0}$
such that $\overline{\Omega_{0}}\subset\Omega$, introduce a cut-off
function $\chi\in\Domega$ such that $\chi=1$ in $\Omega_{0}$. From
\eqref{eq:campa-E1} we deduce 
\begin{equation}
-\div\left(\varepsilon\nabla(\chi E_{k})\right)=\div\left(\chi\left(\partial_{k}\varepsilon\, E+S_{k}\right)\right)+f_{k}\textrm{ in }\Omega,\label{eq:Ek campanato equation-1}
\end{equation}
 where 
\[
f_{k}=-\nabla\chi\cdot\left(\partial_{k}\varepsilon\, E+S_{k}\right)-\varepsilon\nabla E_{k}\cdot\nabla\chi-\div(\varepsilon E_{k}\nabla\chi)\in L^{2}\left(\Omega;\mathbb{C}\right).
\]
 As $\lambda_{1}<2$ and $\varepsilon$ satisfies \eqref{eq:Hyp2-epsilon},
we may apply Proposition~\ref{prop:campanato_regularity} (with $\varepsilon$
in lieu of $\mu$) to show that $\nabla(\chi E_{k})$ is in $L^{2,\lambda_{1}}\left(\Omega;\mathbb{C}^{3}\right)$,
which implies $\nabla E\in L^{2,\lambda_{1}}\left(\Omega_{0};\mathbb{C}^{3\times3}\right)$.

\emph{Boundary regularity.} By using a change of coordinates as in
the proof of Proposition \ref{prop:bootstrap step for E}, we can
assume without loss of generality that $\Omega\cap B(0,R)=\{x\cdot\mathbf{e}_{3}<0\}\cap B(0,R)$.
Indeed, the assumption $\partial\Omega\in\Cl^{1,1}$ implies that
the regularity assumptions on $\epsilon$ and on the source terms
and the condition $\nabla E\in L^{2,\lambda_{1}}$ are insensitive
to a $\Cl^{1,1}$ change of coordinates, as $L^{\infty}$ is a multiplier
space for $L^{2,\lambda_{1}}$.

Let us focus on the tangential components first. Take $\chi\in\mathcal{D}(B(0,R);\mathbb{R})$
such that $\chi=1$ in a neighbourhood $\tilde{B}$ of $0$ and $j\in\{1,2\}$.
Identity \eqref{eq:campa-E1} yields, for $j=1,2$ 
\[
-\div\bigl(\varepsilon\nabla(\chi{E}_{j})\bigr)=\div\left(\chi\left(\partial_{j}\varepsilon\, E+S_{j}\right)\right)+f_{j}\textrm{ in }\Omega,
\]
 where $f_{j}=-\nabla\chi\cdot\left(\partial_{j}\varepsilon\, E+S_{j}\right)-\varepsilon\nabla E_{j}\cdot\nabla\chi-\div(\varepsilon E_{j}\nabla\chi)\in L^{2}(\Omega;\C)$.
Note that $E\times\nu=0$ on $\partial\Omega$ implies $\chi{E}_{1}=\chi{E}_{2}=0$
on $\partial\Omega$. Proposition~\ref{prop:campanato_regularity}
together with \eqref{eq:rhs-est-campa} then shows that $\nabla(\chi{E}_{j})$
belongs to $L^{2,\lambda_{1}}\left(\Omega;\mathbb{C}^{3}\right)$
for $j=1,2$. Arguing as in the proof of Proposition~\ref{prop:bootstrap step for E},
we also derive that $\nabla(\chi{E}{}_{3})\in L^{2,\lambda_{1}}\left(\Omega;\mathbb{C}^{3}\right)$.
Therefore $\nabla(\chi{E})\in L^{2,\lambda_{1}}\bigl(\tilde{B};\mathbb{C}^{3\times3}\bigr)$,
and in turn $\nabla E\in L^{2,\lambda_{1}}\bigl(\tilde{B};\mathbb{C}^{3\times3}\bigr)$.

\emph{Global regularity.} Combining the interior and the boundary
estimates we obtain that $\nabla E$ is in $L^{2,\lambda_{1}}\left(\Omega;\mathbb{C}^{3\times3}\right)$,
together with \eqref{eq:bootstrap campanato}.
\end{proof}
We are now ready to prove the global Hölder regularity result.
\begin{proof}[Proof of Theorem~\ref{thm:global campanato for E}]
 Considering the system satisfied by $E-G$ and $H$, we may assume
$G=0$. Choose any $\tilde{\lambda}>1$ such that $\tilde{\lambda}<\lambda_{{\mu}}$
and $\tilde{\lambda}\le3\frac{p-2}{p}$. Hölder's inequality shows
that $J_{m}\in L^{2,\tilde{\lambda}}(\Omega;\C^{3})$ and $\div J_{e}\in L^{2,\tilde{\lambda}}(\Omega;\C)$.
We apply Proposition~\ref{prop:bootstrap step for E-campanato} a
finite number of times, starting with $\nabla E\in L^{2,\lambda_{n}}\left(\Omega;\mathbb{C}^{3\times3}\right)$
for some $\lambda_{n}<1$ (in the initial step we take $\lambda_{0}=0$,
in view of Theorem~\ref{thm:H1 for E}), and obtain that $\nabla E\in L^{2,\lambda_{n+1}}\left(\Omega;\mathbb{C}^{3\times3}\right)$,
with $\lambda_{n+1}=\min\bigl(\tilde{\lambda},(n+1)2\delta\left(\delta+3\right)^{-1}\bigr)$.
If $\lambda_{n+1}=1$, Proposition~\ref{prop:bootstrap step for E-campanato}
could not be applied another time (as $\lambda_{0}=1$ is excluded).
An easy workaround of course is to reduce $\delta$ to a nearby irrational
(just for this step), and proceed. We stop the iterative procedure
as soon as $\lambda_{n+1}>1$ and we infer that $\nabla E\in L^{2,\lambda}\left(\Omega;\mathbb{C}^{3\times3}\right)$
for some $1<\lambda\le\tilde{\lambda}$. A final application of Proposition~\ref{prop:bootstrap step for E-campanato}
gives $\nabla E\in L^{2,\min\left(\tilde{\lambda},3(1+\delta)(3+\delta)^{-1}\right)}\left(\Omega;\mathbb{C}^{3\times3}\right)$;
the result then follows from Lemma~\ref{lem:campanato_properties}.
\end{proof}
\bibliographystyle{abbrv}
\bibliography{biblio}

\appendix

\section{Proof of Theorem~\ref{thm:C 0 alpha -bi-anisotropic}}

The first step is to derive an appropriate very weak formulation. 
\begin{prop}
\label{prop:coupled_system-bianosotropic} Under the hypotheses of
Theorem~\ref{thm:C 0 alpha -bi-anisotropic}, let $E,H\in H\left(\curl,\Omega\right)$
be a weak solution of \eqref{eq:maxwell for bi-anisotropic-1}.

Then for each $k=1,2,3$, $(E_{k},H_{k})$ is a very weak solution
of the elliptic system 
\[
\left\{ \!\begin{array}{l}
\!\!-\div\left(\varepsilon\nabla E_{k}+\xi\nabla H_{k}\right)=\div\left((\partial_{k}\varepsilon)E+\left(\partial_{k}\xi\right)H-\varepsilon\left(\mathbf{e}_{k}\!\times\!\left(-\ii\k\zeta E-\ii\k\mu H+J_{m}\right)\right)\right)\\
\qquad\qquad\qquad\qquad\quad\;+\div\!\left(-\xi\left(\mathbf{e}_{k}\!\times\!\left(\ii\k\varepsilon E+\ii\k\xi H+J_{e}\right)\right)-\ik\mathbf{e}_{k}\div J_{e}\right)\textrm{ in }\Omega.\\
\!\!-\div\left(\zeta\nabla E_{k}+{\mu}\nabla H_{k}\right)=\div\left((\partial_{k}\zeta)E+\left(\partial_{k}\mu\right)H-{\mu}\left(\mathbf{e}_{k}\times\left(\ii\k\varepsilon E+\ii\k\xi H+J_{e}\right)\right)\right)\\
\qquad\qquad\qquad\qquad\quad\;+\div\!\left(\zeta\left(\mathbf{e}_{k}\!\times\!\left(\ii\k\zeta E+\ii\k\mu H-J_{m}\right)\right)+\ik\mathbf{e}_{k}\div J_{m}\right)\textrm{ in }\Omega.
\end{array}\right.
\]
 More precisely, for any $\varphi\in W^{2,2}(\Omega;\C)$ there holds
\begin{eqnarray}
 & \int_{\Omega}E_{k}\div\left(\tepsi\nabla\overline{\varphi}\right)\, dx+\int_{\Omega}H_{k}\div\left(\xi^{T}\nabla\overline{\varphi}\right)\, dx=\int_{\Omega}(\left(\partial_{k}\varepsilon\right)E+\left(\partial_{k}\xi\right)H)\cdot\nabla\overline{\varphi}\, dx\nonumber \\
 & \qquad-\int_{\Omega}\!\left(\varepsilon\left(\mathbf{e}_{k}\times\left(-\ii\k\zeta E-\ii\k\mu H+J_{m}\right)\right)+\xi\left(\mathbf{e}_{k}\!\times\!\left(\ii\k\varepsilon E+\ii\k\xi H+J_{e}\right)\right)\right)\cdot\nabla\overline{\varphi}\, dx\nonumber \\[-1.5ex]
\label{eq:very weak E bi-anisotropic}\\[-1.5ex]
 & \qquad-\int_{\Omega}\left(\ii\k^{-1}\div J_{e}\mathbf{e}_{k}\right)\cdot\nabla\overline{\varphi}\, dx+\!\int_{\partial\Omega}(\partial_{k}\overline{\varphi})(\varepsilon E+\xi H)\cdot\nu\, d\sigma\nonumber \\
 & -\int_{\partial\Omega}\left(\mathbf{e}_{k}\times\left(H\times\nu\right)\right)\cdot(\xi^{T}\nabla\overline{\varphi})\, d\sigma-\int_{\partial\Omega}\left(\mathbf{e}_{k}\times\left(E\times\nu\right)\right)\cdot(\tepsi\nabla\overline{\varphi})\, d\sigma,\nonumber 
\end{eqnarray}
 and 
\begin{eqnarray}
 & \int_{\Omega}E_{k}\div\left(\zeta^{T}\nabla\overline{\varphi}\right)\, dx+\int_{\Omega}H_{k}\div\left(\tmuu\nabla\overline{\varphi}\right)\, dx=\int_{\Omega}(\left(\partial_{k}\zeta\right)E+\left(\partial_{k}{\mu}\right)H)\cdot\nabla\overline{\varphi}\, dx\nonumber \\
 & \qquad-\int_{\Omega}\left({\mu}\left(\mathbf{e}_{k}\times\left(\ii\k\varepsilon E+\ii\k\xi H+J_{e}\right)\right)-\zeta\left(\mathbf{e}_{k}\times\left(\ii\k\zeta E+\ii\k\mu H-J_{m}\right)\right)\right)\cdot\nabla\overline{\varphi}\, dx\nonumber \\[-1.5ex]
\label{eq:very weak H bi-anisotropic}\\[-1.5ex]
 & \qquad+\int_{\Omega}\left(\ii\k^{-1}\div J_{m}\mathbf{e}_{k}\right)\cdot\nabla\overline{\varphi}\, dx+\int_{\partial\Omega}\!(\partial_{k}\overline{\varphi})(\zeta E+\mu H)\cdot\nu\, d\sigma\nonumber \\
 & -\int_{\partial\Omega}\left(\mathbf{e}_{k}\times\left(E\times\nu\right)\right)\cdot(\zeta^{T}\nabla\overline{\varphi})\, d\sigma-\int_{\partial\Omega}\left(\mathbf{e}_{k}\times\left(H\times\nu\right)\right)\cdot(\tmuu\nabla\overline{\varphi})\, d\sigma.\nonumber 
\end{eqnarray}
 \end{prop}
\begin{proof}
The proof is similar to that of Proposition~\ref{prop:coupled_system}.
\qquad{}
\end{proof}
We only study interior regularity for the problem at hand. The boundary
regularity does not follow easily from the method used in \S~\ref{sec:Main-H1-Ca}.
Indeed, mixed boundary terms appear in \eqref{eq:very weak E bi-anisotropic}
and \eqref{eq:very weak H bi-anisotropic}, and the technique used
in Proposition~\ref{prop:bootstrap step for E} and in Proposition
\ref{prop:bootstrap step for H}, with test functions satisfying either
Dirichlet or Neumann boundary conditions, does not apply, as both
conditions would be required simultaneously.

The ``very weak to weak'' Lemma~\ref{lem:very weak_boundary} adapted
to this mixed system is given below. 
\begin{lem}
\label{lem:very weak_boundary-1}Assume \eqref{eq:strong legendre condition}
and \eqref{eq:Hyp2-epsilon-xi-mu-zeta} hold, and let $A$ be given
by \eqref{eq:matrix A}.

Given $r\geq\frac{6}{5}$, $u\in L^{2}\left(\Omega;\mathbb{R}^{4}\right)\cap L^{r}\left(\Omega;\mathbb{R}^{4}\right)$
and $F\in W^{1,r'}\left(\Omega;\mathbb{R}^{4}\right)'$, if 
\begin{equation}
\int_{\Omega}u^{j}\partial_{\alpha}(A_{ij}^{\alpha\beta}\partial_{\beta}\varphi^{i})\, dx=\langle F_{i},\varphi^{i}\rangle,\qquad\varphi\in W^{2,2}\left(\Omega;\mathbb{R}^{4}\right)\cap W_{0}^{1,2}\left(\Omega;\mathbb{R}^{4}\right),\label{eq:very weak systems}
\end{equation}
 then $u\in W^{1,r}\left(\Omega;\mathbb{R}^{4}\right)$ and 
\begin{equation}
\left\Vert \nabla u\right\Vert _{L^{r}\left(\Omega;\mathbb{R}^{4\times3}\right)}\le C\left\Vert F\right\Vert _{W^{1,r'}\left(\Omega;\mathbb{R}^{4}\right)'},\label{eq:bound norm grad u-systems}
\end{equation}
 for some constant $C=C(r,\Omega,\Lambda,\left\Vert \varepsilon,\xi,\mu,\zeta\right\Vert _{W^{1,3}\left(\Omega;\mathbb{C}^{3\times3}\right)^{4}})$.\end{lem}
\begin{proof}
Let $\psi\in\mathcal{D}\left(\Omega;\mathbb{R}\right)$ be a test
function and take $\alpha^{*}\in\{1,2,3\}$ and $j^{*}\in\{1,\dots,4\}$.
Since $A$ satisfies the strong Legendre condition \eqref{eq:strong legendre condition},
the system 
\begin{equation}
\left\{ \begin{array}{l}
\partial_{\alpha}(A_{ij}^{\alpha\beta}\partial_{\beta}\varphi_{^{*}}^{i})=\delta_{jj^{*}}\partial_{\alpha^{*}}\psi\qquad\mbox{ in }\Omega,\\
\varphi_{*}=0\qquad\mbox{ on \ensuremath{\partial\Omega},}
\end{array}\right.\label{eq:test system}
\end{equation}
 has a unique solution $\varphi_{*}\in H_{0}^{1}\left(\Omega;\mathbb{R}^{4}\right)$
(see e.g. \cite{GIAQUINTA-MARTINAZZI-2005,CHEN-WU-1998}). Further,
since $A_{ij}^{\alpha\beta}\in W^{1,3}(\Omega;\mathbb{R})$, by \cite[Theorem 1.7, Remark 1.8]{BYUN-WANG-2008}
for any $q\in(1,\infty)$ 
\begin{equation}
\left\Vert \varphi_{*}\right\Vert _{W^{1,q}\left(\Omega;\mathbb{R}^{4}\right)}\le c\left\Vert \psi\right\Vert _{L^{q}\left(\Omega;\mathbb{R}^{4}\right)},\label{eq:q bound-systems}
\end{equation}
 for some $c=c(q,\Omega,\Lambda,\left\Vert \varepsilon,\xi,\mu,\zeta\right\Vert _{W^{1,3}\left(\Omega;\mathbb{C}^{3\times3}\right)^{4}})>0$.
Hence, the usual difference quotient argument given in \cite{GIAQUINTA-MARTINAZZI-2005}
shows that $\varphi_{*}\in W^{2,2}\left(\Omega;\mathbb{R}^{4}\right)$.
Therefore, by assumption we have 
\[
\left|\int_{\Omega}\! u^{j^{*}}\!\partial_{\alpha^{*}}\psi\, dx\right|\!=\!\left|\int_{\Omega}\! u^{j}\!\partial_{\alpha}(A_{ij}^{\alpha\beta}\partial_{\beta}\varphi_{^{*}}^{i})\, dx\right|\!=\!\left|\langle F_{i},\varphi_{*}^{i}\rangle\right|\!\le\!\left\Vert F\right\Vert _{W^{1,r'}\!\left(\Omega;\mathbb{R}^{4}\right)'}\!\left\Vert \varphi_{*}\right\Vert _{W^{1,r'}\!\left(\Omega;\mathbb{R}^{4}\right)}\!,
\]
 which in view of \eqref{eq:q bound-systems} gives 
\[
\left|\int_{\Omega}u^{j^{*}}\partial_{\alpha^{*}}\psi\, dx\right|\le c\left\Vert F\right\Vert _{W^{1,r'}\left(\Omega;\mathbb{R}^{4}\right)'}\left\Vert \psi\right\Vert _{L^{r'}\left(\Omega;\mathbb{R}^{4}\right)},
\]
 whence the result. 
\end{proof}
The following proposition mirrors Propositions \ref{prop:bootstrap step for E}
and \ref{prop:bootstrap step for H}. Theorem~\ref{thm:C 0 alpha -bi-anisotropic}
then follows by the bootstrap argument used in the proof of Theorem~\ref{thm:W^1,p for E and H}. 
\begin{prop}
\label{prop:bootstrap step for E,H-bi-anisotropic} Under the hypotheses
of Theorem~\ref{thm:C 0 alpha -bi-anisotropic} and given $q\in[2,\infty)$,
set $r=\min((3q+q\delta)(q+3+\delta)^{^{-1}},p)$. Let $E$ and $H$
in $\Hcurl$ be weak solutions of \eqref{eq:maxwell for bi-anisotropic-1}.

Suppose $E,H\in L^{q}(\Omega;\C^{3})$. Then $E,H\in W_{loc}^{1,r}(\Omega;\C^{3})$
and for any open subdomain $\Omega_{0}$ such that $\overline{\Omega_{0}}\subset\Omega$
there holds 
\begin{equation}
\left\Vert (E,H)\right\Vert _{W^{1,r}(\Omega_{0};\C^{3})}\le C(\left\Vert (E,H)\right\Vert _{L^{q}(\Omega;\C)^{6}}+\left\Vert (J_{e},J_{m})\right\Vert _{W^{1,p}(\div,\Omega)^{2}}),\label{eq:bound norm E,H in W^1,p - systems}
\end{equation}
 for some constant $C=C(r,\Omega,\Omega_{0},\Lambda,\k,\left\Vert \varepsilon,\xi,\mu,\zeta\right\Vert _{W^{1,\delta+3}\left(\Omega;\mathbb{C}^{3\times3}\right)^{4}})$. \end{prop}
\begin{proof}
From \eqref{eq:very weak E bi-anisotropic} we see that for every
compactly supported $\varphi^{1},\varphi^{2}\in W^{2,2}(\Omega;\C)$
and $k=1,2,3$ there holds 
\begin{equation}
\left\{ \begin{array}{l}
\int_{\Omega}E_{k}\div\left(\tepsi\nabla\overline{\varphi^{1}}\right)+H_{k}\div\left(\xi^{T}\nabla\overline{\varphi^{1}}\right)\, dx=\int_{\Omega}F_{k}\cdot\nabla\overline{\varphi^{1}}\, dx,\\
\int_{\Omega}E_{k}\div\left(\zeta^{T}\nabla\overline{\varphi^{2}}\right)+H_{k}\div\left(\tmuu\nabla\overline{\varphi^{2}}\right)\, dx=\int_{\Omega}G_{k}\cdot\nabla\overline{\varphi^{2}}\, dx,
\end{array}\right.\label{eq:weak for bi-anisotropic}
\end{equation}
 with 
\begin{eqnarray*}
F_{k} & = & \left(\partial_{k}\varepsilon\right)E+\left(\partial_{k}\xi\right)H-\varepsilon\left(\mathbf{e}_{k}\times\left(-\ii\k\zeta E-\ii\k\mu H+J_{m}\right)\right)\\
 &  & -\xi\left(\mathbf{e}_{k}\times\left(\ii\k\varepsilon E+\ii\k\xi H+J_{e}\right)\right)-\ii\k^{-1}\div J_{e}\mathbf{e}_{k},
\end{eqnarray*}
 and 
\begin{eqnarray*}
G_{k} & = & \left(\partial_{k}\zeta\right)E+\left(\partial_{k}\mu\right)H-{\mu}\left(\mathbf{e}_{k}\times\left(\ii\k\varepsilon E+\ii\k\xi H+J_{e}\right)\right)\\
 &  & +\zeta\left(\mathbf{e}_{k}\times\left(\ii\k\zeta E+\ii\k{\mu}H-J_{m}\right)\right)+\ii\k^{-1}\div J_{m}\mathbf{e}_{k}.
\end{eqnarray*}
 By construction, $F_{k},G_{k}\in L^{r}(\Omega;\C^{3})$.

Given a smooth subdomain $\Omega_{0}$, we consider a cut-off function
$\chi\in\Domega$ such that $\chi=1$ in $\Omega_{0}$. A straightforward
computation shows 
\[
\left\{ \begin{array}{l}
\int_{\Omega}\chi E_{k}\div(\tepsi\nabla\overline{\varphi^{1}})+\chi H_{k}\div\left(\xi^{T}\nabla\overline{\varphi^{1}}\right)\, dx=\int_{\Omega}F_{k}\cdot\nabla(\chi\overline{\varphi^{1}})\, dx+T_{k}(\varphi^{1}),\\
\int_{\Omega}\chi E_{k}\div(\zeta^{T}\nabla\overline{\varphi^{2}})+\chi H_{k}\div\left({\mu}^{T}\nabla\overline{\varphi^{2}}\right)\, dx=\int_{\Omega}G_{k}\cdot\nabla(\chi\overline{\varphi^{2}})\, dx+R_{k}(\varphi^{2}),
\end{array}\right.
\]
 where 
\[
T_{k}(\varphi^{1})=-\!\int_{\Omega}\! E_{k}\bigl(\div(\tepsi\overline{\varphi^{1}}\nabla\chi)+\varepsilon\nabla\chi\cdot\nabla\overline{\varphi^{1}}\bigr)+H_{k}\bigl(\div(\xi^{T}\overline{\varphi^{1}}\nabla\chi)+\xi\nabla\chi\cdot\nabla\overline{\varphi^{1}}\bigr)\, dx,
\]
 and 
\[
R_{k}(\varphi^{2})=-\!\int_{\Omega}\! E_{k}\bigl(\div(\zeta^{T}\overline{\varphi^{2}}\nabla\chi)+\zeta\nabla\chi\cdot\nabla\overline{\varphi^{2}}\bigr)+H_{k}\bigl(\div(\tmuu\overline{\varphi^{2}}\nabla\chi)+{\mu}\nabla\chi\cdot\nabla\overline{\varphi^{2}}\bigr)\, dx.
\]
 This last system can be reformulated in the form \eqref{eq:very weak systems},
with $A$ given by \eqref{eq:matrix A}. We then apply Lemma~\ref{lem:very weak_boundary-1}
and obtain $\chi E_{k},\chi H_{k}\in W^{1,r}(\Omega;\C)$, namely
$E,H\in W^{1,r}(\Omega_{0};\C^{3})$. Finally, \eqref{eq:bound norm E,H in W^1,p - systems}
follows from \eqref{eq:bound norm grad u-systems}.\end{proof}

\end{document}